\documentclass[12pt,draftcls,onecolumn]{IEEEtran}

\def\bsg{{\boldsymbol{g}}}

\def\bsu{{\boldsymbol{u}}}
\def\bsv{{\boldsymbol{v}}}

\def\bsx{{\boldsymbol{x}}}

\def\bsH{{\boldsymbol{H}}}

\def\bsK{{\boldsymbol{K}}}
\def\bsL{{\boldsymbol{L}}}

\def\bsM{{\boldsymbol{M}}}

\def\bsQ{{\boldsymbol{Q}}}

\def\calA{{\mathcal{A}}}

\def\calF{{\mathcal{F}}}
\def\calG{{\mathcal{G}}}

\let\labelindent\relax
\usepackage{amsmath}
\usepackage{amssymb}
\usepackage{algorithm,algorithmic}
\usepackage[table]{xcolor}
\usepackage[short]{optidef}

\usepackage{mathtools, cuted}
\usepackage{wrapfig}
\usepackage{pifont}
\usepackage{silence}
\usepackage{tikz}
\usepackage{tkz-berge}
\usepackage{lipsum}
\usetikzlibrary{calc}
\usetikzlibrary{shapes,fit}
\usepackage{verbatim}
\usepackage{setspace}
\usetikzlibrary{shapes,arrows}
\usepackage{color}

\usepackage{amsthm}

\usepackage{color}
\definecolor{gray}{RGB}{128,128,128}

\usepackage{cite}

\usepackage{graphicx}
\graphicspath{{figure/}}
\usepackage{epstopdf}
\usepackage{tikz}
\usetikzlibrary{arrows,shapes,backgrounds}

\usepackage{bm}

\usepackage{array}
\newcolumntype{M}[1]{>{\centering\arraybackslash}m{#1}}
\newcolumntype{N}{@{}m{0pt}@{}}

\usepackage{multirow}

\usepackage{caption}
\usepackage{subcaption}

\newtheorem{theorem}{Theorem}
\newtheorem{assumption}{Assumption}

\newtheorem{lemma}{Lemma}
\newtheorem{corollary}{Corollary}

\newtheorem{remark}{Remark}

\DeclareMathOperator{\nullrank}{null}

\DeclareMathOperator{\col}{col}

\DeclareMathOperator{\diag}{diag}

\newenvironment{proof}[1][Proof]%
  {\smallskip\par\noindent\textbf{#1\,:\ }}%
  {\hspace*{\fill} \rule{6pt}{6pt}\smallskip}
\newenvironment{proof*}[1][Proof]%
  {\smallskip\par\noindent\textbf{#1\,:\ }}%

\newlength{\fwidth}

\allowdisplaybreaks
\begin{document}
\IEEEoverridecommandlockouts
\title{A Primal--Dual SGD Algorithm  for\\ Distributed Nonconvex Optimization}
\author{Xinlei Yi, Shengjun Zhang, Tao Yang, Tianyou Chai, and Karl H. Johansson
\thanks{This work was supported by the
Knut and Alice Wallenberg Foundation, the  Swedish Foundation for Strategic Research, the Swedish Research Council, and the National Natural Science Foundation of China under grants 61991403, 61991404, and 61991400.}
\thanks{X. Yi and K. H. Johansson are with the Division of Decision and Control Systems, School of Electrical Engineering and Computer Science, KTH Royal Institute of Technology, 100 44, Stockholm, Sweden. {\tt\small \{xinleiy, kallej\}@kth.se}.}%
\thanks{S. Zhang is with the Department of Electrical Engineering, University of North Texas, Denton, TX 76203 USA. {\tt\small  ShengjunZhang@my.unt.edu}.}
\thanks{T. Yang and T. Chai are with the State Key Laboratory of Synthetical Automation for Process Industries, Northeastern University, 110819, Shenyang, China. {\tt\small \{yangtao,tychai\}@mail.neu.edu.cn}.}
}

\maketitle

\begin{abstract}                
The distributed nonconvex optimization problem of minimizing a global cost function formed by a sum of $n$ local cost functions by using local information exchange is considered.
This problem is an important component of many machine learning techniques with data parallelism, such as deep learning and federated learning.
We propose a distributed primal--dual stochastic gradient descent (SGD) algorithm,  suitable for arbitrarily connected communication networks and any smooth (possibly nonconvex) cost functions.
We show that the proposed algorithm achieves the linear speedup convergence rate $\mathcal{O}(1/\sqrt{nT})$ for general nonconvex cost functions and the linear speedup convergence rate $\mathcal{O}(1/(nT))$ when the global cost function satisfies the Polyak--{\L}ojasiewicz (P--{\L}) condition, where $T$ is the total number of iterations.
We also show that the output of the proposed algorithm with constant parameters linearly converges to a neighborhood of a global optimum.
We demonstrate through numerical experiments the efficiency of our algorithm in comparison with the baseline centralized SGD and recently proposed distributed SGD algorithms.

\end{abstract}


\section{Introduction}
Consider a network of $n$ agents, each of which has a local smooth (possibly nonconvex) cost function $f_i: \mathbb{R}^{p}\rightarrow \mathbb{R}$.
All agents collaboratively solve the following optimization problem
\begin{align}\label{sgd:eqn:xopt}
 \min_{x\in \mathbb{R}^p} f(x)=\frac{1}{n}\sum_{i=1}^nf_i(x).
\end{align}
Each agent $i$ only has information about its local cost function $f_i$ and can communicate with its neighbors through the underlying communication network. The communication network is modeled by an undirected graph $\mathcal G=(\mathcal V,\mathcal E)$, where $\mathcal V =\{1,\dots,n\}$ is the agent set, $\mathcal E
\subseteq \mathcal V \times \mathcal V$ is the edge set, and $(i,j)\in \mathcal E$ if agents $i$ and $j$ can communicate with each other. The set $\mathcal{N}_i=\{j\in \mathcal V :~ (i,j)\in \mathcal E\}$ is the neighboring set of agent $i$.
The optimization problem \eqref{sgd:eqn:xopt} incorporates many popular machine learning approaches with data parallelism, such as deep learning \cite{dean2012large} and federated learning \cite{McMahan2017communication}.
A star graph is a special undirected graph, in which there is one and only one agent  (hub agent) that connects to all of the other agents (leaf agents) and each leaf agent only connects to the hub agent. Such a graph corresponds to the master/worker architecture adopted by many parallel learning algorithms.

In this paper, we consider the case where each agent is able to collect stochastic gradients of its local cost function and propose a distributed stochastic gradient descent (SGD) algorithm to solve  \eqref{sgd:eqn:xopt}.
In general, SGD algorithms are suitable for scenarios where explicit expressions of the gradients are unavailable or difficult to obtain.
For example, in some big data applications, such as empirical risk minimization, the actual gradient is to be calculated from the entire data set, which results in a heavy computational burden.
A stochastic gradient can be calculated from a randomly selected subset of the data and is often an efficient way to replace the actual gradient.
Other examples when SGD algorithms are suitable include scenarios where data are arriving sequentially such as in online learning \cite{langford2009sparse}.


\subsection{Literature Review}
When the communication network is a star graph, various parallel SGD algorithms have been proposed to solve \eqref{sgd:eqn:xopt}.
A potential performance bottleneck of such algorithms lies on the communication burden of the master.
To overcome this issue, a promising strand of research is combining parallel SGD algorithms with communication reduction approaches, e.g., asynchronous parallel SGD algorithms \cite{recht2011hogwild,de2015taming,lian2015asynchronous,lian2016Comprehensive,Zhou2018distributedas},
gradient compression based parallel SGD algorithms \cite{de2015taming,pmlr-v80-bernstein18a,
jiang2018linear,reisizadeh2019fedpaq,basu2019qsparse},
periodic averaging based parallel SGD algorithms \cite{jiang2018linear,reisizadeh2019fedpaq,wang2018adaptive,yu2019parallel,
haddadpour2019trading,Yu2019on,haddadpour2019local}, and parallel SGD algorithm with dynamic batch sizes \cite{Yu2019Computation}.
Convergence properties of these algorithms have been analyzed in detail.
In particular, in \cite{jiang2018linear,yu2019parallel,Yu2019on,Yu2019Computation}, an $\mathcal{O}(1/\sqrt{nT})$ convergence rate has been established for general nonconvex cost functions, where $T$ is the total number of iterations.
This rate is $n$ times faster than the well known $\mathcal{O}(1/\sqrt{T})$ convergence rate established by SGD over a single agent, and thus a linear speedup in the number of agents is achieved.
In \cite{haddadpour2019local,Yu2019Computation}, the  convergence rate has been improved to $\mathcal{O}(1/(nT))$ when the global cost function satisfies the P--{\L} condition, which also achieves a linear speedup.
In addition to the star architecture restriction, aforementioned parallel SGD algorithms require certain restrictions on the cost functions, such as bounded gradients of the local cost functions or bounded difference between the gradients of the local and global cost functions.

Distributed algorithms executed over arbitrarily connected communication networks have been suggested to overcome communication bottlenecks for parallel SGD algorithms.
Various distributed SGD algorithms have been proposed to solve \eqref{sgd:eqn:xopt}, e.g., synchronous distributed SGD algorithms \cite{Yu2019on,jiang2017collaborative,lian2017can,george2019distributed},
asynchronous distributed SGD algorithms \cite{pmlr-v80-lian18a,Assran2019Stochastic},
compression based distributed SGD algorithms \cite{tang2018communication,reisizadeh2019robust,taheri2020quantized,Singh2020Communication},
and periodic averaging based distributed SGD algorithm \cite{wang2018cooperative}.
Convergence properties of these algorithms have been analyzed and the linear speedup convergence rate $\mathcal{O}(1/\sqrt{nT})$ has been established for general nonconvex cost functions \cite{Yu2019on,lian2017can,Assran2019Stochastic,tang2018communication,taheri2020quantized,
Singh2020Communication,wang2018cooperative}.
However, similar to aforementioned parallel SGD algorithms, these distributed algorithms require  restrictive assumptions on the cost functions.
In order to remove these restrictions, the authors of \cite{Tang2018Decentralized} proposed a variant of the distributed SGD algorithm proposed in \cite{lian2017can}, named $\mathrm{D}^2$, in which each agent stores the stochastic gradient and its local model in last iteration and linearly combines them with the current stochastic gradient and local model. For this algorithm the authors established the linear speedup convergence rate $\mathcal{O}(1/\sqrt{nT})$, but they required that the eigenvalues of the mixing matrix associated with the communication network  are strictly greater than $-1/3$.
The authors of \cite{lu2019gnsd,zhang2019decentralized} proposed  distributed stochastic gradient tracking algorithms suitable for arbitrarily connected communication networks.
However, these algorithms only achieve $\mathcal{O}(1/\sqrt{T})$ convergence rate, which is not a speedup.
Moreover, gradient tracking algorithms have the common potential drawback that in order to track the global gradient,
at each iteration each agent needs to communicate one additional $p$-dimensional variable with its neighbors.
This results in heavy communication  burden when $p$ is large.
Note that all aforementioned distributed SGD algorithms converge to stationary points, which may be local or global optima, or saddle points.
None of existing studies on distributed SGD algorithms consider finding the global optimum when the global cost function satisfies some additional property, such as the P--{\L} condition studied for the parallel algorithms in  \cite{haddadpour2019local,Yu2019Computation}. Noting above, two core theoretical questions arise.

(Q1) Are there any distributed SGD algorithms such that they not only are suitable for arbitrarily connected communication networks and any smooth cost functions but also find stations points with the linear speedup convergence rate $\mathcal{O}(1/\sqrt{nT})$?

(Q2) If the P--{\L} condition holds in addition, can the above SGD algorithms find the global optimum with the linear speedup convergence rate $\mathcal{O}(1/(nT))$?

\subsection{Main Contributions}
This paper provides positive answers for the above two questions. More specifically, the contributions of this paper are summarized as follows.

(i) We propose a distributed primal--dual SGD algorithm to solve the optimization problem~\eqref{sgd:eqn:xopt}.
In the proposed algorithm, each agent maintains the primal and dual variable sequences and only communicates the primal variable with its neighbors.
This algorithm is suitable for arbitrarily connected communication networks and any smooth (possibly nonconvex) cost functions.

(ii) We show in Corollary~\ref{sguT:coro-sg-smT} that our algorithm finds a stationary point with the linear speedup convergence rate $\mathcal{O}(1/\sqrt{nT})$ for general nonconvex cost functions, thus answers (Q1). Compared with \cite{jiang2018linear,yu2019parallel,Yu2019on,Yu2019Computation,lian2017can,
Assran2019Stochastic,tang2018communication,taheri2020quantized,Singh2020Communication,
wang2018cooperative,Tang2018Decentralized}, we achieve the same convergence rate but under weaker assumptions related to network architectures and/or cost functions, and compared with \cite{lu2019gnsd,zhang2019decentralized}, we not only establish linear speedup but also just use half communication in each iteration.

(iii) We show in Theorem~\ref{sgu:thm-sg-diminishingt} that our algorithm finds a global optimum with the linear speedup convergence rate $\mathcal{O}(1/(nT))$   when the global cost function satisfies the P--{\L} condition,  thus answers (Q2).  Compared with \cite{haddadpour2019local,Yu2019Computation,stich2018local,Koloskova2019decentralized,
Singh2020Communication,olshevsky2019non} , we achieve the same convergence rate but under weaker assumptions related to network architectures and/or cost functions, and compared with \cite{reisizadeh2019fedpaq,jiang2017collaborative,rabbat2015multi,lan2018communication,
yuan2018optimal,jakovetic2018convergence,fallah2019robust}, we not only establish linear speedup but also relax the strong convexity by the P--{\L} condition.

(iv) We show in Theorems~\ref{sg:thm-sg-fixed-a5} and \ref{sg:thm-sg-fixed} that the output of our algorithm with constant parameters linearly converges to a neighborhood of a global optimum when the global cost function satisfies the P--{\L} condition.
Compared with \cite{pu2018swarming,jiang2017collaborative,fallah2019robust,pu2018distributed,xin2019distributed}, which used the strong convexity assumption, we achieve the similar convergence result under weaker assumptions on the cost function.

The comparison of this paper with other related studies in the literature is summarized in Table~\ref{sgd:table}.
\begin{table*}[htbp]
\caption{Comparison of this paper to some related works. }
\label{sgd:table}
\vskip -1in
\begin{center}
\begin{tiny}
\scalebox{1}{
\begin{tabular}{M{0.9cm}|M{1.0cm}|M{1.9cm}|M{1.9cm}|M{2.2cm}|M{1.8cm}|M{4.8cm}N}
\hline
Reference&Problem type&Extra assumption&Communication network&Communicated variable&Communication rounds&Convergence rate&\\[7pt]

\hline
\cite{jiang2018linear}&Nonconvex&Bounded $\|\nabla f_i-\nabla f\|$&Star graph&One quantized variable &$\mathcal{O}(n^{5/4}T^{3/4})$&$\mathcal{O}(1/\sqrt{nT})$&\\[9pt]

\hline
\multirow{2}{*}[-4.5pt]{\cite{reisizadeh2019fedpaq}}& Nonconvex& \multirow{2}{*}[-4.5pt]{Identical $\nabla f_i$}
& \multirow{2}{*}[-4.5pt]{Star graph}& \multirow{2}{*}[-4.5pt]{\parbox{2.2cm}{\centering One quantized variable}}& \multirow{2}{*}[-4.5pt]{$\mathcal{O}(T)$}&$\mathcal{O}(1/\sqrt{T})$&\\[9pt]
\cline{2-2}\cline{7-7}
&Strongly convex& & & & &$\mathcal{O}(1/T)$&\\[9pt]

\hline
\cite{yu2019parallel}&Nonconvex&Bounded $\|\nabla f_i\|$&Star graph&One full-information variable &$\mathcal{O}(n^{3/4}T^{3/4})$&$\mathcal{O}(1/\sqrt{nT})$&\\[9pt]

\hline
\multirow{2}{*}[-4.5pt]{\cite{Yu2019on}}&  \multirow{2}{*}[-4.5pt]{Nonconvex}& \multirow{2}{*}[-4.5pt]{\parbox{1.9cm}{\centering Bounded $\|\nabla f_i-\nabla f\|$}}
&Star graph& \multirow{2}{*}[-4.5pt]{\parbox{2.2cm}{\centering Two full-information variables}}& $\mathcal{O}(n^{3/4}T^{3/4})$&\multirow{2}{*}[-4.5pt]{$\mathcal{O}(1/\sqrt{nT})$}&\\[9pt]
\cline{4-4}\cline{6-6}
& & &Connected graph & & $\mathcal{O}(T)$&&\\[9pt]

\hline
\cite{haddadpour2019local}&P--{\L} condition&Identical $\nabla f_i$&Star graph&One full-information variable &$\mathcal{O}((nT)^{1/3})$&$\mathcal{O}(1/(nT))$&\\[9pt]

\hline
\multirow{2}{*}[-4.5pt]{\cite{Yu2019Computation}}&Nonconvex&\multirow{2}{*}[1.5pt]{\parbox{1.9cm}{\centering Identical $\nabla f_i$, exponentially increasing batch size}}&\multirow{2}{*}[-4.5pt]{Star graph}&\multirow{2}{*}[-4.5pt]{\parbox{2.2cm}{\centering One full-information variable}} &$\mathcal{O}(\sqrt{nT}\log(\frac{T}{n}))$&$\mathcal{O}(1/\sqrt{nT})$&\\[9pt]
\cline{2-2}\cline{6-7}
&P--{\L} condition& & & &$\mathcal{O}(\log(T))$&$\mathcal{O}(1/(nT))$&\\[9pt]

\hline
\multirow{2}{*}[-4.5pt]{\cite{jiang2017collaborative}}&Nonconvex&\multirow{2}{*}[-4.5pt]{Bounded $\|\nabla f_i\|$}&\multirow{2}{*}[-4.5pt]{Connected graph}&\multirow{2}{*}[-4.5pt]{\parbox{2.2cm}{\centering One full-information variable}} &\multirow{2}{*}[-4.5pt]{$\mathcal{O}(T)$}&$\mathcal{O}(1/T^\theta),~\forall \theta\in(0,0.5)$&\\[9pt]
\cline{2-2}\cline{7-7}
&Strongly convex& & & &&$\mathcal{O}(1/T)$; linearly to a neighbor of the global optimum (constant stepsize)&\\[9pt]

\hline
\cite{lian2017can}&Nonconvex&Bounded $\|\nabla f_i-\nabla f\|$&Connected graph&One full-information variable &$\mathcal{O}(T)$&$\mathcal{O}(1/\sqrt{nT})$&\\[9pt]

\hline
\cite{Assran2019Stochastic}&Nonconvex&Bounded $\|\nabla f_i-\nabla f\|$&Uniformly jointly
strongly connected digraph&One full-information variable &$\mathcal{O}(T)$&$\mathcal{O}(1/\sqrt{nT})$&\\[9pt]

\hline
\cite{tang2018communication}&Nonconvex&Bounded $\|\nabla f_i-\nabla f\|$&Connected graph&One compressed variable &$\mathcal{O}(T)$&$\mathcal{O}(1/\sqrt{nT})$&\\[9pt]

\hline
\cite{taheri2020quantized}&Nonconvex&Bounded $\|\nabla f_i\|$&Strongly connected digraph&One quantized variable &$\mathcal{O}(T)$&$\mathcal{O}(1/\sqrt{nT})$&\\[9pt]

\hline
\multirow{2}{*}[-4.5pt]{\cite{Singh2020Communication}}&Nonconvex&\multirow{2}{*}[-4.5pt]{Bounded $\|\nabla f_i\|$}&\multirow{2}{*}[-4.5pt]{Connected graph}&\multirow{2}{*}[-4.5pt]{\parbox{2.2cm}{\centering One compressed variable}} &\multirow{2}{*}[-4.5pt]{Event-triggered}&$\mathcal{O}(1/\sqrt{nT})$&\\[9pt]
\cline{2-2}\cline{7-7}
&Strongly convex& & & &&$\mathcal{O}(1/(nT))$&\\[9pt]

\hline
\cite{wang2018cooperative}&Nonconvex&Identical $\nabla f_i$&Connected graph&One full-information variable &$\mathcal{O}(n^{3/2}\sqrt{T})$&$\mathcal{O}(1/\sqrt{nT})$&\\[9pt]

\hline
\cite{Tang2018Decentralized}&Nonconvex&The eigenvalues of the mixing matrix are strictly greater than $-1/3$&Connected graph&One full-information variable &$\mathcal{O}(T)$&$\mathcal{O}(1/\sqrt{nT})$&\\[9pt]

\hline
\cite{lu2019gnsd,zhang2019decentralized}&Nonconvex&No&Connected graph&Two full-information variables &$\mathcal{O}(T)$&$\mathcal{O}(1/\sqrt{T})$&\\[9pt]

\hline
\cite{stich2018local}&Strongly convex&Bounded $\|\nabla f_i\|$&Star graph&One full-information variable &$\mathcal{O}(\sqrt{T/n})$&$\mathcal{O}(1/(nT))$&\\[9pt]

\hline
\cite{Koloskova2019decentralized}&Strongly convex&Bounded $\|\nabla f_i\|$&Connected graph&One compressed variable &$\mathcal{O}(T)$&$\mathcal{O}(1/(nT))$&\\[9pt]

\hline
\cite{olshevsky2019non}&Strongly convex&No&Connected graph&Two full-information variables &$\mathcal{O}(T)$&$\mathcal{O}(1/(nT))$&\\[9pt]

\hline
\cite{rabbat2015multi}&Strongly convex&Identical $\nabla f_i$&Connected graph&One full-information variable &$\mathcal{O}(T)$&$\mathcal{O}(1/T)$&\\[9pt]

\hline
\cite{lan2018communication}&Strongly convex&No&Connected graph&One full-information variable &$\mathcal{O}(\sqrt{T})$&$\mathcal{O}(1/T)$&\\[9pt]

\hline
\cite{yuan2018optimal}&Strongly convex&Bounded $\|\nabla f_i\|$&Uniformly jointly
strongly connected digraph&One full-information variable &$\mathcal{O}(T)$&$\mathcal{O}(1/T)$&\\[9pt]

\hline
\cite{jakovetic2018convergence}&Strongly convex&No&Connected graph in expectation&One full-information variable &$\mathcal{O}(T)$&$\mathcal{O}(1/T)$&\\[9pt]

\hline
\cite{fallah2019robust}&Strongly convex&No&Connected graph&One full-information variable &$\mathcal{O}(T)$&$\mathcal{O}(1/T)$; linearly to a neighbor of the global optimum (constant stepsize)&\\[9pt]

\hline
\cite{pu2018swarming}&Strongly convex&No&Connected graph&One full-information variable &$\mathcal{O}(T)$&Linearly to a neighbor of the global optimum (constant stepsize)&\\[9pt]

\hline
\cite{pu2018distributed}&Strongly convex&No&Connected graph&Two full-information variables &$\mathcal{O}(T)$&Linearly to a neighbor of the global optimum (constant stepsize)&\\[9pt]

\hline
\cite{xin2019distributed}&Strongly convex&No&Strongly connected digraph&Two full-information variables &$\mathcal{O}(T)$&Linearly to a neighbor of the global optimum (constant stepsize)&\\[9pt]

\hline
\multirow{3}{*}[-12pt]{\parbox{0.9cm}{\centering This paper}}&Nonconvex&\multirow{2}{*}[-14pt]{No}&\multirow{3}{*}[-14pt]{Connected graph}&\multirow{3}{*}[-14pt]{\parbox{2.2cm}{\centering One full-information variable}} &\multirow{3}{*}[-14pt]{$\mathcal{O}(T)$}&$\mathcal{O}(1/\sqrt{nT})$&\\[9pt]
\cline{2-2}\cline{7-7}
&\multirow{2}{*}[-9pt]{\parbox{1cm}{\centering P--{\L} condition}}&& & &&$~~~~~~~~~$$\mathcal{O}(1/(T^\theta)),~\forall \theta\in(0,1)$;$~~~~~~~~~$ linearly to a neighbor of the global optimum (constant stepsize)&\\[9pt]
\cline{3-3}\cline{7-7}
&&Bounded $f^*_i$ & & &&$\mathcal{O}(1/(nT))$&\\[9pt]
\hline
\end{tabular}
}
\end{tiny}
\end{center}
\vskip -0.1in
\end{table*}

\subsection{Outline}
The rest of this paper is organized as follows. Section~\ref{sguT:sec-main} presents the new distributed primal--dual SGD algorithm. Section~\ref{sguT:sec-analysis} analyzes its convergence rate. Numerical experiments are given in Section~\ref{sgd:sec-simulation}. Finally, concluding remarks are offered in Section~\ref{sgd:sec-conclusion}.
To improve the readability, all the proofs are given in the appendix.

\noindent {\bf Notations}: $\mathbb{N}_0$ and $\mathbb{N}_+$ denote the set of nonnegative and positive integers, respectively. $[n]$ denotes the set $\{1,\dots,n\}$ for any $n\in\mathbb{N}_+$. ${\bm 1}_n$ (${\bm 0}_n$) denotes the column one (zero) vector of dimension $n$. $\col(z_1,\dots,z_k)$ is the concatenated column vector of vectors $z_i\in\mathbb{R}^{p_i},~i\in[k]$.
$\|\cdot\|$ represents the Euclidean norm for vectors or the induced 2-norm for matrices. Given a differentiable function $f$, $\nabla f$ denotes the gradient of $f$.

\section{Distributed Primal--Dual SGD Algorithm}\label{sguT:sec-main}

In this section, we propose a new distributed SGD algorithm based on the primal--dual method.

Denote $\bsx=\col(x_1,\dots,x_n)$, $\tilde{f}(\bsx)=\sum_{i=1}^{n}f_i(x_i)$,  and $\bsL=L\otimes {\bm I}_p$, where $L=(L_{ij})$ is the weighted Laplacian matrix associated with the undirected communication graph $\calG$.
Recall that the Laplacian matrix $L$ is positive semi-definite and $\nullrank(L)=\{{\bm 1}_n\}$ when $\calG$ is connected \cite{mesbahi2010graph}. The optimization problem \eqref{sgd:eqn:xopt} is equivalent to the following constrained optimization problem:
\begin{mini}
{\bsx\in \mathbb{R}^{np}}{\tilde{f}(\bsx)}{\label{sgd:eqn:xoptcon}}{}
\addConstraint{\bsL^{1/2}\bsx=}{~{\bm 0}_{np}.}{}
\end{mini}

Let $\bsu\in\mathbb{R}^{np}$ denote the dual variable. Then the augmented Lagrangian function associated with \eqref{sgd:eqn:xoptcon} is
\begin{align}\label{nonconvex:lagran}
\calA(\bsx,\bsu)=\tilde{f}(\bsx)+\frac{\alpha}{2}\bsx^\top\bsL\bsx+\beta\bsu^\top\bsL^{1/2}\bsx,
\end{align}
where $\alpha>0$ and $\beta>0$ are parameters to be designed later.

Based on the primal--dual gradient method, a distributed SGD algorithm to solve \eqref{sgd:eqn:xoptcon} is
\begin{subequations}\label{sgu:kiau-algo-compact}
\begin{align}
\bm{x}_{k+1}&=\bm{x}_k-\eta_k(\alpha_k\bsL\bm{x}_k+\beta_k\bsL^{1/2}\bm{u}_k+\bsg_k^u),\\
\bm{u}_{k+1}&=\bm{u}_{k}+\eta_k\beta_k\bsL^{1/2}\bm{x}_k,~\forall \bsx_0,~\bsu_0\in\mathbb{R}^{np},
\end{align}
\end{subequations}
where $\eta_k>0$ is the stepsize at iteration $k$, $\alpha_k>0$ and $\beta_k>0$ are the values of the parameters $\alpha$ and $\beta$ at iteration $k$, respectively, and $\bsg_k^u=\col(g^u_{1,k},\dots,g^u_{n,k})$ with $g^u_{i,k}=g_{i}(x_{i,k},\xi_{i,k})$ being the stochastic gradient of $f_i$ at $x_{i,k}$ and $\xi_{i,k}$ being a random variable.
Denote $\bsv_k=\col(v_{1,k},\dots,v_{n,k})=\bsL^{1/2}\bm{u}_k$. Then the recursion \eqref{sgu:kiau-algo-compact} can be rewritten as
\begin{subequations}\label{sgu:kia-algo-dc-compact}
\begin{align}
\bm{x}_{k+1}&=\bm{x}_k-\eta_k(\alpha_k\bsL\bm{x}_k+\beta_k\bm{v}_k+\bsg_k^u),
\label{sgu:kia-algo-dc-compact-x}\\
\bm{v}_{k+1}&=\bm{v}_k+\eta_k\beta_k\bsL\bm{x}_k,~\forall \bsx_0\in\mathbb{R}^{np},~\sum_{j=1}^nv_{j,0}={\bm 0}_p.\label{sgu:kia-algo-dc-compact-v}
\end{align}
\end{subequations}
The initialization condition $\sum_{j=1}^nv_{j,0}={\bm 0}_p$ is derived from $\bsv_0=\bsL^{1/2}\bm{u}_0$, and it is easy to be satisfied, for example, $v_{i,0}={\bm 0}_p,~\forall i\in[n]$, or $v_{i,0}=\sum_{j=1}^{n}L_{ij}x_{j,0},~\forall i\in[n]$.
Note that \eqref{sgu:kia-algo-dc-compact} can  be written agent-wise as
\begin{subequations}\label{sguT:alg:stochastic}
\begin{align}
x_{i,k+1} &= x_{i,k}-\eta_k\Big(\alpha_k\sum_{j=1}^nL_{ij}x_{j,k}+\beta_k v_{i,k}+g^u_{i,k}\Big), \label{sguT:alg:stochastic-x}\\
v_{i,k+1} &=v_{i,k}+ \eta_k\beta_k\sum_{j=1}^n L_{ij}x_{j,k},~\forall x_{i,0}\in\mathbb{R}^p, ~v_{i,0}={\bm 0}_p,~
\forall i\in[n].  \label{sguT:alg:stochastic-q}
\end{align}
\end{subequations}
This corresponds to our proposed distributed primal--dual SGD algorithm, which is presented in pseudo-code as Algorithm~\ref{sguT:algorithm-sg}.
\begin{algorithm}[tb]
\caption{Distributed Primal--Dual SGD Algorithm}
\label{sguT:algorithm-sg}
\begin{algorithmic}[1]
\STATE \textbf{Input}: parameters $\{\alpha_k\}$, $\{\beta_k\}$, $\{\eta_k\}\subseteq(0,+\infty)$.
\STATE \textbf{Initialize}: $ x_{i,0}\in\mathbb{R}^p$ and $v_{i,0}={\bm 0}_p,~
\forall i\in[n]$.
\FOR{$k=0,1,\dots$}
\FOR{$i=1,\dots,n$  in parallel}
\STATE  Broadcast $x_{i,k}$ to $\mathcal{N}_i$ and receive $x_{j,k}$ from $j\in\mathcal{N}_i$;
\STATE  Sample stochastic gradient $g_{i}(x_{i,k},\xi_{i,k})$;
\STATE  Update $x_{i,k+1}$ by \eqref{sguT:alg:stochastic-x};
\STATE  Update $v_{i,k+1}$ by \eqref{sguT:alg:stochastic-q}.
\ENDFOR
\ENDFOR
\STATE  \textbf{Output}: $\{\bsx_k\}$.
\end{algorithmic}
\end{algorithm}

It should be pointed out that $\{\alpha_k\}$, $\{\beta_k\}$, $\{\eta_k\}$, $\bsx_0$, $\bsv_0$, and $\bsv_1$  in Algorithm~\ref{sguT:algorithm-sg} are deterministic, while $\{\bsx_k\}_{k\ge1}$ and $\{\bsv_k\}_{k\ge2}$ are random variables generated by Algorithm~\ref{sguT:algorithm-sg}. Let $\mathfrak{F}_k$ denote the $\sigma$-algebra generated by the random variables $\xi_{1,k},\dots,\xi_{n,k}$ and let $\mathcal{F}_k=\bigcup_{s=1}^{k}\mathfrak{F}_s$. It is straightforward to see that $\bsx_k$ and $\bsv_{k+1}$ depend on $\mathcal{F}_{k-1}$ and are independent of $\mathfrak{F}_s$ for all $s\ge k$.

\section{Convergence Rate Analysis}\label{sguT:sec-analysis}
In this section, we analyze the convergence rate of Algorithm~\ref{sguT:algorithm-sg}.
The following assumptions are made.

\begin{assumption}\label{sgd:assgraph}
The undirected communication graph $\mathcal G$ is connected.
\end{assumption}

\begin{assumption}\label{sgd:assoptset}
The minimum function value of the optimization problem \eqref{sgd:eqn:xopt} is finite.
\end{assumption}

\begin{assumption}\label{sgd:assfiu}
Each  local cost function $f_i$ is smooth with constant $L_{f}>0$, i.e.,
\begin{align}\label{nonconvex:smooth}
\|\nabla f_i(x)-\nabla f_i(y)\|\le L_{f}\|x-y\|,~\forall x,y\in \mathbb{R}^p.
\end{align}
\end{assumption}

\begin{assumption}\label{sgd:ass:stochastic-grad:xi}
The random variables $\{\xi_{i,k},~i\in[n],~k\in\mathbb{N}_0\}$ are independent of each other.
\end{assumption}

\begin{assumption}\label{sgd:ass:stochastic-grad:mean}
The stochastic estimate $g_{i}(x,\xi_{i,k})$ is unbiased, i.e., for all $i\in[n]$, $k\in\mathbb{N}_0$, and $x\in\mathbb{R}^p$,
\begin{align}\label{sgd:ass:stochastic-grad:mean-equ}
\mathbf{E}_{\xi_{i,k}}[g_{i}(x,\xi_{i,k})]=\nabla f_{i}(x).
\end{align}
\end{assumption}

\begin{assumption}\label{sgd:ass:stochastic-grad:variance}
The stochastic estimate $g_{i}(x,\xi_{i,k})$ has bounded variance, i.e., there exists a constant $\sigma$ such that for all $i\in[n]$, $k\in\mathbb{N}_0$, and $x\in\mathbb{R}^p$,
\begin{align}\label{sgd:ass:stochastic-grad:variance-equ}
\mathbf{E}_{\xi_{i,k}}[\|g_{i}(x,\xi_{i,k})-\nabla f_{i}(x)\|^2]\le\sigma^2.
\end{align}
\end{assumption}

\begin{remark}
The bounded variance assumption (Assumption~\ref{sgd:ass:stochastic-grad:variance}) is weaker than the bounded second moment (or bounded gradient) assumption made in \cite{recht2011hogwild,de2015taming,lian2015asynchronous,Zhou2018distributedas,basu2019qsparse,
yu2019parallel,jiang2017collaborative,george2019distributed,taheri2020quantized,yuan2018optimal,
rakhlin2011making,stich2018local,Koloskova2019decentralized,Singh2020Communication}. Moreover, note that we make no assumption on the boundedness of the deviation between the gradients of local cost functions. In other words, we do not assume that $\frac{1}{n}\sum_{i=1}^{n}\|\nabla f_i(x)-\nabla f(x)\|^2$ is uniformly bounded, which is commonly done in studies of deep learning, e.g., \cite{jiang2018linear,yu2019parallel,Yu2019on,lian2017can,pmlr-v80-lian18a,Assran2019Stochastic,tang2018communication}.
Also, we do not assume that the mean of each local stochastic gradient is the gradient of the global cost function, i.e.,  $\mathbf{E}_{\xi}[g_{i}(x,\xi)]=\nabla f(x),~\forall x\in\mathbb{R}^p,~\forall i\in[n]$,  which is commonly assumed in studies of empirical risk minimization and stochastic optimization, e.g., \cite{lian2016Comprehensive,pmlr-v80-bernstein18a,reisizadeh2019fedpaq,wang2018adaptive,haddadpour2019trading,
haddadpour2019local,Yu2019Computation,reisizadeh2019robust,wang2018cooperative,rabbat2015multi}.
\end{remark}

\subsection{Find Stationary Points}

Let us consider the case when Algorithm~\ref{sguT:algorithm-sg} is able to find stationary points.
We have the following convergence results.

\begin{theorem}\label{sguT:thm-sg-smT}
Suppose Assumptions~\ref{sgd:assgraph}--\ref{sgd:ass:stochastic-grad:variance} hold. Let $\{\bsx_k\}$ be the sequence generated by Algorithm~\ref{sguT:algorithm-sg} with
\begin{align}\label{sguT:step:eta2-sm}
\alpha_k=\kappa_1\beta_k,~\beta_k=\beta,~ \eta_k=\frac{\kappa_2}{\beta_k},~\forall k\in\mathbb{N}_0,
\end{align}
where $\kappa_1>c_1$, $\kappa_2\in(0,c_2(\kappa_1))$, and $\beta\ge c_0(\kappa_1,\kappa_2)$ with $c_0(\kappa_1,\kappa_2),~c_1,~c_2(\kappa_1)>0$  defined in Appendix~\ref{sguT:proof-thm-sg-smT}.
Then, for any $T\in\mathbb{N}_+$,
\begin{subequations}
\begin{align}
&\frac{1}{T}\sum_{k=0}^{T-1}\mathbf{E}\Big[\frac{1}{n}\sum_{i=1}^{n}\|x_{i,k}-\bar{x}_k\|^2\Big]
=\mathcal{O}(\frac{1}{T})
+\mathcal{O}(\frac{1}{\beta^2}),\label{sguT:thm-sg-sm-equ3.1}\\
&\frac{1}{T}\sum_{k=0}^{T-1}\mathbf{E}[\|\nabla f(\bar{x}_k)\|^2]
=\mathcal{O}(\frac{\kappa_2\beta}{T})
+\mathcal{O}(\frac{4\kappa_2}{n\beta})
+\mathcal{O}(\frac{1}{T})+\mathcal{O}(\frac{1}{\beta^2}),\label{sguT:thm-sg-sm-equ3}\\
&\mathbf{E}[f(\bar{x}_{T})]-f^*=\mathcal{O}(1)
+\mathcal{O}(\frac{T}{n\beta^2})+\mathcal{O}(\frac{T}{\beta^3}),\label{sguT:thm-sg-sm-equ4}
\end{align}
\end{subequations}
where $f^*$ is the minimum function value of the optimization problem \eqref{sgd:eqn:xopt} and $\bar{x}_k=\frac{1}{n}\sum_{i=1}^{n}x_{i,k}$.
\end{theorem}
\begin{proof}
The explicit expressions of the right-hand sides of \eqref{sguT:thm-sg-sm-equ3.1}--\eqref{sguT:thm-sg-sm-equ4} and the proof are given in Appendix~\ref{sguT:proof-thm-sg-smT}. It should be highlighted that the omitted constants in the first two terms in the right-hand side of  \eqref{sguT:thm-sg-sm-equ3} do not depend on any parameters related to the communication network.
\end{proof}

Noting the right-hand side of \eqref{sguT:thm-sg-sm-equ3} and $f(\bar{x}_0)-f^*=\mathcal{O}(1)$, the linear speedup (w.r.t. number of agents) can be established if we set $\beta=\kappa_2\sqrt{T}/\sqrt{n}$, as shown in the following.

\begin{corollary}[\bf Linear Speedup]\label{sguT:coro-sg-smT}
Under the same assumptions as in Theorem~\ref{sguT:thm-sg-smT}, let $\beta=\kappa_2\sqrt{T}/\sqrt{n}$. Then, for any $T>  \max\{n(c_0(\kappa_1,\kappa_2)/\kappa_2)^2,~n^3\}$,
\begin{subequations}
\begin{align}
&\frac{1}{T}\sum_{k=0}^{T-1}\mathbf{E}\Big[\frac{1}{n}\sum_{i=1}^{n}\|x_{i,k}-\bar{x}_k\|^2\Big]
=\mathcal{O}(\frac{n}{T}),\label{sguT:coro-sg-sm-equ3.1}\\
&\frac{1}{T}\sum_{k=0}^{T-1}\mathbf{E}[\|\nabla f(\bar{x}_k)\|^2]
=\mathcal{O}(\frac{1}{\sqrt{nT}})+\mathcal{O}(\frac{n}{T}),\label{sguT:coro-sg-sm-equ3}\\
&\mathbf{E}[f(\bar{x}_{T})]-f^*=\mathcal{O}(1).\label{sguT:coro-sg-sm-equ4}
\end{align}
\end{subequations}
\end{corollary}

\begin{remark}
It should be highlighted that the omitted constants in the first term in the right-hand side of  \eqref{sguT:coro-sg-sm-equ3} do not depend on any parameters related to the communication network.
The same linear speedup result as in \eqref{sguT:coro-sg-sm-equ3} was also established by the SGD algorithms proposed in \cite{jiang2018linear,yu2019parallel,Yu2019on,Yu2019Computation,lian2017can,Assran2019Stochastic,
tang2018communication,taheri2020quantized,wang2018cooperative,Tang2018Decentralized,Singh2020Communication}.
However, in \cite{jiang2018linear,Yu2019on,lian2017can,Assran2019Stochastic,tang2018communication},
the additional assumption that the deviation between the gradients of local cost functions is bounded was made;
in \cite{yu2019parallel,taheri2020quantized,Singh2020Communication}, it was required that each local stochastic gradient has bounded second moment;
in \cite{Yu2019Computation,wang2018cooperative}, it was assumed that the mean of each local stochastic gradient is the gradient of the global cost function;
and in \cite{Tang2018Decentralized}, it was  required that the eigenvalues of the mixing matrix are strictly greater than $-1/3$. Moreover, the algorithms proposed in \cite{jiang2018linear,Yu2019Computation} are restricted to a star graph;
the distributed momentum SGD algorithm proposed in \cite{Yu2019on} requires each agent $i$ to communicate one additional $p$-dimensional variable besides the communication of $x_{i,k}$ with its neighbors at each iteration;
and the algorithm proposed in \cite{Yu2019Computation} requires an exponentially increasing batch size, which is not favorable in practice.
Under the same conditions,  the well known $\mathcal{O}(1/\sqrt{T})$ convergence rate, which is not a speedup, was achieved by the distributed stochastic gradient tracking algorithm proposed in \cite{lu2019gnsd,zhang2019decentralized}.
Moreover, similar to the distributed momentum SGD algorithm proposed in \cite{Yu2019on}, one potential drawback of the distributed stochastic gradient tracking algorithms is that at each iteration each agent needs to communicate one additional variable.
The potential drawbacks of the results stated in Corollary~\ref{sguT:coro-sg-smT} are that (i) we do not consider communication efficiency, which was considered in \cite{jiang2018linear,yu2019parallel,Yu2019Computation,tang2018communication,taheri2020quantized,
wang2018cooperative,Singh2020Communication}; and (ii) we use time-invariant undirected graphs rather than directed graphs as considered in \cite{Assran2019Stochastic,taheri2020quantized}. We leave the extension to the time-varying directed graphs with communication efficiency as future research directions.
\end{remark}

\subsection{Find Global Optimum}

Let us next consider cases when Algorithm~\ref{sguT:algorithm-sg} finds global optimum.
The following assumption is crucial.

\begin{assumption}\label{sgd:assfil} The global cost function $f(x)$ satisfies the Polyak--{\L}ojasiewicz (P--{\L}) condition with constant $\nu>0$, i.e.,
\begin{align}
\frac{1}{2}\|\nabla f(x)\|^2\ge \nu( f(x)-f^*),~\forall x\in \mathbb{R}^p.\label{nonconvex:equ:plc}
\end{align}
\end{assumption}

It is straightforward to see that every (essentially or weakly) strongly convex function satisfies the P--{\L} condition.
The P--{\L} condition implies that every stationary point is a global minimizer. But unlike (essentially or weakly) strong convexity, the P--{\L} condition  alone does not imply convexity of $f$. Moreover, it does not imply that the global minimizer is a unique either \cite{karimi2016linear,zhang2015restricted}.

Many practical applications, such as least squares and logistic regression, do not always have strongly convex cost functions.
The cost function in least squares problems has the form
\begin{align*}
f(x)=\frac{1}{2}\|Ax-b\|^2,
\end{align*}
where $A\in\mathbb{R}^{m\times p}$ and $b\in\mathbb{R}^m$. Note that if $A$ has full column rank, then $f(x)$ is strongly convex. However, if $A$ is rank deficient, then $f(x)$ is not strongly convex, but it is convex and satisfies the P--{\L} condition.
Examples of nonconvex functions which satisfy the P--{\L} condition can be found in \cite{karimi2016linear,zhang2015restricted}.

Although it is difficult to precisely characterize the general class of functions for which the P--{\L} condition is satisfied, in \cite{karimi2016linear}, one important special class was given as follows:
\begin{lemma}
Let $f(x)=g(Ax)$, where $g:\mathbb{R}^p\rightarrow\mathbb{R}$ is a strongly convex function and $A\in\mathbb{R}^{p\times p}$ is a matrix, then $f$ satisfies the P--{\L} condition.
\end{lemma}

We have the following global convergence results.

\begin{theorem}\label{sguT:thm-sg-diminishingT}
Suppose Assumptions~\ref{sgd:assgraph}--\ref{sgd:assfil} hold. For any given $T\ge(c_0(\kappa_1,\kappa_2)/\kappa_2)^{1/\theta}$, let $\{\bsx_0,\dots,\bsx_T\}$ be the sequence generated by Algorithm~\ref{sguT:algorithm-sg} with
\begin{align}\label{sguT:step:eta1T}
\alpha_k=\kappa_1\beta_k,~\beta_k=\kappa_2(T+1)^\theta,
~ \eta_k=\frac{\kappa_2}{\beta_k},~\forall k\le T,
\end{align}
where $\theta\in(0,1)$, $\kappa_1>c_1$, $\kappa_2\in(0,c_2(\kappa_1))$.
Then,
\begin{subequations}
\begin{align}
&\mathbf{E}\Big[\frac{1}{n}\sum_{i=1}^{n}\|x_{i,T}-\bar{x}_T\|^2\Big]
=\mathcal{O}(\frac{1}{T^{2\theta}}), \label{sguT:thm-sg-equ2.1bounded}\\
&\mathbf{E}[f(\bar{x}_{T})-f^*]
=\mathcal{O}(\frac{1}{nT^\theta})+\mathcal{O}(\frac{1}{T^{2\theta}}).
\label{sguT:thm-sg-equ2bounded}
\end{align}
\end{subequations}
\end{theorem}
\begin{proof}
The explicit expressions of the right-hand sides of \eqref{sguT:thm-sg-equ2.1bounded} and \eqref{sguT:thm-sg-equ2bounded}, and the proof are given in Appendix~\ref{sguT:proof-thm-sg-diminishingT}. It should be highlighted that the omitted constants in the first term in the right-hand side of \eqref{sguT:thm-sg-equ2bounded} do not depend on any parameters related to the communication network.
\end{proof}

From Theorem~\ref{sguT:thm-sg-diminishingT}, we see that the convergence rate is strictly greater than $\mathcal{O}(1/(nT))$. In the following we show that the linear speedup convergence rate $\mathcal{O}(1/(nT))$ can be achieved if the P--{\L} constant $\nu$ is known in advance and each $f_i^*>-\infty$, where $f_i^*=\min_{x\in\mathbb{R}^p}f_i(x)$. The total number of iterations $T$ is not needed.
\begin{theorem}[\bf Linear Speedup]\label{sgu:thm-sg-diminishingt}
Suppose Assumptions~\ref{sgd:assgraph}--\ref{sgd:assfil} hold, and the P--{\L} constant $\nu$ is known in advance, and each $f_i^*>-\infty$.
Let $\{\bsx_{k}\}$ be the sequence generated by Algorithm~\ref{sguT:algorithm-sg} with \begin{align}\label{sgu:step:eta1t1}
\alpha_k=\kappa_1\beta_k,~\beta_k=\kappa_0(k+t_1),~ \eta_k=\frac{\kappa_2}{\beta_k},~\forall k\in\mathbb{N}_0,
\end{align}
where $\kappa_0\in[\hat{c}_0\nu\kappa_2/4,\nu\kappa_2/4)$, $\kappa_1>c_1$, $\kappa_2\in(0,\hat{c}_2(\kappa_1))$, and $t_1>\hat{c}_3(\kappa_0,\kappa_1,\kappa_2)$ with $\hat{c}_0\in(0,1)$ being a constant, $\hat{c}_2(\kappa_1)$ and $\hat{c}_3(\kappa_0,\kappa_1,\kappa_2)$ defined in Appendix~\ref{sgu:proof-thm-sg-diminishingt}.
Then, for any $T\in\mathbb{N}_+$,
\begin{subequations}
\begin{align}
&\mathbf{E}\Big[\frac{1}{n}\sum_{i=1}^{n}\|x_{i,T}-\bar{x}_T\|^2\Big]
=\mathcal{O}(\frac{1}{T^{2}}), \label{sgu:thm-sg-diminishing-equ2.1bounded}\\
&\mathbf{E}[f(\bar{x}_{T})-f^*]
=\mathcal{O}(\frac{1}{nT})+\mathcal{O}(\frac{1}{T^{2}}).
\label{sgu:thm-sg-diminishing-equ2bounded}
\end{align}
\end{subequations}
\end{theorem}
\begin{proof}
The explicit expressions of the right-hand sides of \eqref{sgu:thm-sg-diminishing-equ2.1bounded} and \eqref{sgu:thm-sg-diminishing-equ2bounded}, and the proof are given  in Appendix~\ref{sgu:proof-thm-sg-diminishingt}. It should be highlighted that the omitted constants in the first term in the right-hand side of  \eqref{sgu:thm-sg-diminishing-equ2bounded} do not depend on any parameters related to the communication network.
\end{proof}
\begin{remark}
It has been shown in \cite{rakhlin2011making} that $\mathcal{O}(1/T)$ convergence rate is optimal for centralized strongly convex optimization. This rate has been established by various distributed SGD algorithms when each local cost function is strongly convex, e.g., \cite{reisizadeh2019fedpaq,jiang2017collaborative,rabbat2015multi,lan2018communication,
yuan2018optimal,jakovetic2018convergence,fallah2019robust}. In contrast, the linear speedup convergence rate $\mathcal{O}(1/(nT))$ established in Theorem~\ref{sgu:thm-sg-diminishingt} only requires that the global cost function satisfies the P--{\L} condition, but no convexity assumption is required neither on the global cost function nor on the local cost functions. The SGD algorithms in \cite{haddadpour2019local,Yu2019Computation,stich2018local,Koloskova2019decentralized,
Singh2020Communication,olshevsky2019non} also achieve the linear speedup convergence rate.
However, the algorithms in \cite{haddadpour2019local,Yu2019Computation,stich2018local} are restricted to a star graph, while our algorithm is applicable to an arbitrarily connected graph.
Moreover, \cite{haddadpour2019local,Yu2019Computation} assumed that the mean of each local stochastic gradient is the gradient of the global cost function, and $T$ has to be known to choose the algorithm parameters. The algorithm in \cite{Yu2019Computation} furthermore requires an exponentially increasing batch size, which is not favorable in practice. In \cite{stich2018local}, it was assumed that the global cost function is strongly convex. In \cite{stich2018local,Singh2020Communication}, it was assumed that  each local stochastic gradient has bounded second moment. In \cite{Koloskova2019decentralized,Singh2020Communication,
olshevsky2019non}, it was assumed that each local cost function is strongly convex. It is one of our future research directions to achieve linear speedup with reduced communication rounds and communication efficiency for an  arbitrarily connected graph.
\end{remark}


Theorem~\ref{sgu:thm-sg-diminishingt} shows that the convergence rate to a global optimum is sublinear when we allow the algorithm parameters to be time-varying.
The following theorem establishes that the output of Algorithm~\ref{sguT:algorithm-sg} with constant algorithm parameters linearly converges to a neighborhood of a global optimum.

\begin{theorem}\label{sg:thm-sg-fixed-a5}
Suppose Assumptions~\ref{sgd:assgraph}--\ref{sgd:assfil} hold. Let $\{\bsx_{k}\}$ be the sequence generated by Algorithm~\ref{sguT:algorithm-sg} with \begin{align}\label{sgu:step:fixed-a5}
\alpha_k=\alpha=\kappa_1\beta,~\beta_k=\beta,~ \eta_k=\eta=\frac{\kappa_2}{\beta},~\forall k\in\mathbb{N}_0,
\end{align}
where $\kappa_1>c_1$, $\kappa_2\in(0,c_2(\kappa_1))$, and $\beta\ge c_0(\kappa_1,\kappa_2)$ with $c_0(\kappa_1,\kappa_2),~c_1,~c_2(\kappa_1)>0$  defined in Appendix~\ref{sguT:proof-thm-sg-smT}. Then,
\begin{align}\label{sg:thm-sg-fixed-equ1-a5}
&\mathbf{E}\Big[\frac{1}{n}\sum_{i=1}^{n}\|x_{i,k}-\bar{x}_k\|^2+f(\bar{x}_k)-f^*\Big]
\le(1-\eta\varepsilon)^{k}c_4+c_5\eta\sigma^2,~\forall k\in\mathbb{N}_+,
\end{align}
where $\varepsilon\in(0,1/\eta),~c_4,~c_5>0$ are constants defined in Appendix~\ref{sg:proof-thm-sg-fixed-a5}.
\end{theorem}
\begin{proof}
The proof is given in Appendix~\ref{sg:proof-thm-sg-fixed-a5}.
\end{proof}
\begin{remark}
It should be highlighted that we  do not need to know the P--{\L} constant $\nu$ in advance. Similar convergence result as stated in \eqref{sg:thm-sg-fixed-equ1-a5} was achieved by the distributed SGD algorithms proposed in \cite{jiang2017collaborative,fallah2019robust,pu2018swarming,pu2018distributed,xin2019distributed} when each local cost function is strongly convex, which obviously is stronger than the P--{\L} condition assumed in Theorem~\ref{sg:thm-sg-fixed-a5}. In addition to the strong convexity condition, in \cite{jiang2017collaborative}, it was also assumed that each local cost function is Lipschitz-continuous. Some information related to the Lyapunov function and global parameters, which may be difficult to get, were furthermore needed to design the stepsize. Moreover, in \cite{fallah2019robust,pu2018swarming,pu2018distributed,xin2019distributed}, the strong convexity constant was needed to design the stepsize and in \cite{pu2018distributed,xin2019distributed}, a $p$-dimensional auxiliary variable, which is used to track the global gradient, was communicated between agents. The potential drawbacks of the results stated in Theorem~\ref{sg:thm-sg-fixed-a5} are that (i) we use undirected graphs rather than directed graphs as considered in \cite{xin2019distributed}; and (ii) we do not analyze the robustness level to gradient noise as \cite{fallah2019robust} did. We leave the extension to the (time-varying) directed graphs and the robustness level analysis as future research directions.
\end{remark}

Actually, the unbiased assumption, i.e., Assumption~\ref{sgd:ass:stochastic-grad:mean}, can be removed, as shown in the following.

\begin{theorem}[\bf Biased SGD]\label{sg:thm-sg-fixed}
Suppose Assumptions~\ref{sgd:assgraph}--\ref{sgd:ass:stochastic-grad:xi}, \ref{sgd:ass:stochastic-grad:variance}, and \ref{sgd:assfil} hold. Let $\{\bsx_{k}\}$ be the sequence generated by Algorithm~\ref{sguT:algorithm-sg} with
\begin{align}\label{sgu:step:fixed}
\alpha_k=\alpha=\kappa_1\beta,~\beta_k=\beta,~ \eta_k=\eta=\frac{\kappa_2}{\beta},~\forall k\in\mathbb{N}_0,
\end{align}
where $\kappa_1>c_1$, $\kappa_2\in(0,c_2(\kappa_1))$, and $\beta\ge \breve{c}_0(\kappa_1,\kappa_2)$ with $\breve{c}_0(\kappa_1,\kappa_2)>0$ and $c_1,~c_2(\kappa_1)>0$  defined in Appendices~\ref{sg:proof-thm-sg-fixed} and \ref{sguT:proof-thm-sg-smT}, respectively. Then,
\begin{align}\label{sg:thm-sg-fixed-equ1}
&\mathbf{E}[\frac{1}{n}\sum_{i=1}^{n}\|x_{i,k}-\bar{x}_k\|^2+f(\bar{x}_k)-f^*]
\le(1-\eta\varepsilon)^{k}c_4+\breve{c}_5\sigma^2,~\forall k\in\mathbb{N}_+,
\end{align}
where $\varepsilon\in(0,1/\eta),~c_4>0$ and $\breve{c}_5>0$ are constants defined in Appendices~\ref{sg:proof-thm-sg-fixed-a5} and \ref{sg:proof-thm-sg-fixed}, respectively.
\end{theorem}
\begin{proof}
The proof is given in Appendix~\ref{sg:proof-thm-sg-fixed}.
\end{proof}

\begin{remark}
By comparing \eqref{sg:thm-sg-fixed-equ1-a5} with \eqref{sg:thm-sg-fixed-equ1}, we can see that no matter the unbiased assumption holds or not,  the output of Algorithm~\ref{sguT:algorithm-sg} with constant algorithm parameters linearly converges to a neighborhood of a global optimum, but the size of neighborhood is different. Specifically, in \eqref{sg:thm-sg-fixed-equ1-a5} the size of neighborhood is in an order of $\mathcal{O}(\eta)$, while it is $\mathcal{O}(1)$ in \eqref{sg:thm-sg-fixed-equ1}. When true gradients are available, i.e., $\sigma=0$, then from \eqref{sg:thm-sg-fixed-equ1-a5} or \eqref{sg:thm-sg-fixed-equ1} we know that a global
optimum can be linearly found. It should be highlighted that this linear convergence result is established under the P--{\L} condition and the P--{\L} constant is not used. These are two advantages since in existing studies obtaining linear convergence for distributed smooth optimization, e.g., \cite{shi2015extra,nedic2017achieving,li2020revisiting}, it is standard to assume (restricted) strong convexity, which is stronger than the P--{\L} condition, and to use the convexity parameter.
\end{remark}

\section{Simulations}\label{sgd:sec-simulation}


In this section, we evaluate the performance of the proposed distributed primal--dual SGD algorithm through numerical experiments.
All algorithms and agents are implemented and simulated in MATLAB R2018b, run on a desktop with \texttt{Intel Core i5-9600K processor, Nvidia RTX 2070 super, 32~GB RAM, Ubuntu 16.04}.

\subsection{Neural Networks}
We consider the training of neural networks (NN) for image classification tasks of the database MNIST \cite{lecun2010mnist}.
The same NN is adopted as in \cite{george2019distributed} for each agent and the communication graph is generated randomly. The graph is shown in Fig.~\ref{sgd:fig:graph} and the corresponding Laplacian matrix $\bm L$ is given in \eqref{sgd:lap}. The corresponding mixing matrix $\bm W$ is constructed by metropolis weight, which is given in \eqref{sgd:lapW}.

\begin{align} \label{sgd:lap}
\bm L =
\begin{bmatrix}
1 & -1 & 0 & 0 & 0 & 0 & 0 & 0 & 0 & 0\\
-1 & 3 & -1 & -1 & 0 & 0 & 0 & 0 & 0 & 0\\
0 & -1 & 3 & -1 & 0 & 0 & -1 & 0 & 0 & 0\\
0 & -1 & -1 & 4 & -1 & -1 & 0 & 0 & 0 & 0\\
0 & 0 & 0 & -1 & 2 & -1 &  0 & 0 & 0 & 0\\
0 & 0 & 0 & -1 & -1 & 2 &  0 & 0 & 0 & 0\\
0 & 0 & -1 & 0 & 0 & 0 &  2 & -1 & 0 & 0\\
0 & 0 & 0 & 0 & 0 & 0 &  -1 & 2 & -1 & 0\\
0 & 0 & 0 & 0 & 0 & 0 & 0 &  -1 & 2 & -1 \\
0 & 0 & 0 & 0 & 0 & 0 & 0 &  0 & -1 & 1\\
\end{bmatrix}.
\end{align}
\begin{align}\label{sgd:lapW}
 \bm W =
\begin{bmatrix}
3/4 & 1/4 & 0 & 0 & 0 & 0 & 0 & 0 & 0 & 0\\
1/4 & 3/10 & 1/4 & 1/5 & 0 & 0 & 0 & 0 & 0 & 0\\
0 & 1/4 & 3/10 & 1/5 & 0 & 0 & 1/4 & 0 & 0 & 0\\
0 & 1/5 & 1/5 & 1/5 & 1/5 & 1/5 & 0 & 0 & 0 & 0\\
0 & 0 & 0 & 1/5 & 7/15 & 1/3 &  0 & 0 & 0 & 0\\
0 & 0 & 0 & 1/5 & 1/3 & 7/15 &  0 & 0 & 0 & 0\\
0 & 0 & 1/4 & 0 & 0 & 0 &  5/12 & 1/3 & 0 & 0\\
0 & 0 & 0 & 0 & 0 & 0 &  1/3 & 1/3 & 1/3 & 0\\
0 & 0 & 0 & 0 & 0 & 0 & 0 &  1/3 & 1/3 & 1/3 \\
0 & 0 & 0 & 0 & 0 & 0 & 0 &  0 & 1/3 & 2/3\\
\end{bmatrix}.
\end{align}

Each local neural network consists of a single hidden layer of 50 neurons, followed by a sigmoid activation layer, followed by the output layer of 10 neurons and another sigmoid activation layer. In this experiment, we use a subset of MNIST data set. Each agent is assigned 2500 data points randomly, and at each iteration, only one data point is picked up by the agent following a uniform distribution.

We compare our proposed distributed primal--dual SGD algorithm with time-varying and fixed  parameters (DPD-SGD-T and DPD-SGD-F)  with state-of-the-art algorithms: distributed momentum SGD algorithm (DM-SGD) \cite{Yu2019on},  distributed SGD algorithm (D-SGD-1) \cite{jiang2017collaborative,lian2017can},  distributed SGD algorithm (D-SGD-2) \cite{george2019distributed},  $\mathrm{D}^2$ \cite{Tang2018Decentralized}, distributed stochastic gradient tracking algorithm (D-SGT-1)  \cite{lu2019gnsd,xin2019distributed},  distributed stochastic gradient tracking algorithm  (D-SGT-2) \cite{zhang2019decentralized,pu2018distributed}, and  the baseline centralized SGD algorithm (C-SGD). We list all the parameters\footnote{Note: the parameter names are different in each paper.} we choose in the NN experiment for each algorithm in Table~\ref{tab:nn-par}.

We demonstrate the result in terms of the empirical risk function \cite{bottou2012stochastic}, which is given as
\begin{align*}
R(\bm z) &= - \frac{1}{n}\sum_{i = 1}^{n}\frac{1}{m_{n}} \sum _{j=1}^{m_{n}} \sum _{k=0}^{9}(t_{k} \ln y_k(\bm{x}, \bm z) + (1-t_{k})\ln (1 - y_k(\bm{x}, \bm z)))
\end{align*}
where $m_{n}$ indicates the size of data set for each agent, $t_k$ denotes the target (ground truth) of digit $k$ corresponding to a single image, $\bm x$ is a single image input, $\bm z=(z^{(1)}, z^{(2)})$ with $z^{(1)}$ and $z^{(2)}$ being the weights in the 2 layers separately, and $y_k \in [0, 1]$ is the output which expresses the probability of digit $k = 0, \dots, 9$.
The mapping from input to output is given as:
\begin{align*}
y_k(\bm{x}, \bm z) = \sigma\left( \sum_{j=0}^{50} z_{k, j}^{(2)} \sigma \left( \sum_{i = 0}^{28 \times 28} z_{j, i}^{(1)}x_{i}\right)\right),
\end{align*}
where $\sigma (s) = \frac{1}{1+\exp(-s)}$ is the sigmoid function.

\begin{figure}
\centering
\hspace{10mm}\centering{
{
\begin{tikzpicture}[-,node distance=1.4cm,
  thick,main node/.style={circle,fill=yellow!20,draw,font=\sffamily\normalsize\bfseries}]
  \node[main node] (1) {1};
  \node[main node] (2) [above right of=1] {2};
  \node[main node] (3) [right of=2] {3};
  \node[main node] (4) [below right of=2] {4};
  \node[main node] (5) [below right of=1] {5};
  \node[main node] (6) [ below of=5] {6};
  \node[main node] (7) [ right of=3] {7};
  \node[main node] (8) [ below of=7] {8};
  \node[main node] (9) [ below of=8] {9};
  \node[main node] (10) [ left of=9] {10};
\draw (1) -- (2)
(2) -- (3)
(3) -- (7)
(2) -- (4)
(3) -- (4)
(4) -- (5)
(4) -- (6)
(5) -- (6)
(7) -- (8)
(8) -- (9)
(9) -- (10);
\end{tikzpicture}}}$\qquad\qquad$
\vspace{1mm}
\caption{\label{sgd:fig:graph} Connection Topology.}
\end{figure}
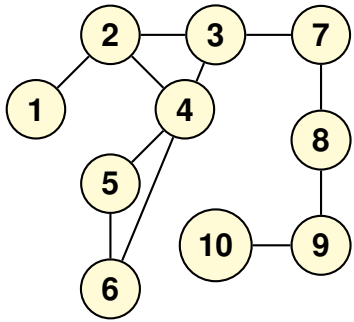

\begin{table*}[htbp]
\caption{{Parameters in each algorithm in NN experiment.} }
\label{tab:nn-par}
\vskip 0.15in
\begin{center}
\begin{normalsize}
\begin{tabular}{M{3.5cm}|M{3.5cm}|M{3.5cm}|M{3.5cm}}
\hline
Algorithm & $\eta_k$ & $\alpha_k$ & $\beta_k$ \\

\hline
DPD-SGD-T & $0.08/{k^{10^{-5}}}$ & $4k^{10^{-5}}$ & $3k^{10^{-5}}$\\

\hline
DPD-SGD-F &0.03 & 5 & 20\\

\hline
DM-SGD \cite{Yu2019on} & 0.1 & \ding{55} &0.8\\

\hline
D-SGD-1 \cite{jiang2017collaborative,lian2017can}& 0.1 &\ding{55} &\ding{55}\\

\hline
D-SGD-2 \cite{george2019distributed} &\ding{55} & $0.1/(10^{-5}k + 1)$&$0.2/(10^{-5}k + 1)^{0.3}$\\

\hline
$D^{2}$ \cite{Tang2018Decentralized}& 0.01 &\ding{55} &\ding{55}\\

\hline
D-SGT-1 \cite{lu2019gnsd,xin2019distributed} & 0.01& \ding{55}&\ding{55}\\

\hline
D-SGT-2 \cite{zhang2019decentralized,pu2018distributed} & 0.01 & \ding{55}&\ding{55}\\

\hline
C-SGD & 0.1 &\ding{55} &\ding{55}\\
\hline
\end{tabular}
\end{normalsize}
\end{center}
\vskip -0.1in
\end{table*}

\begin{figure}[htbp]
\centering
  \includegraphics[width=0.85\textwidth]{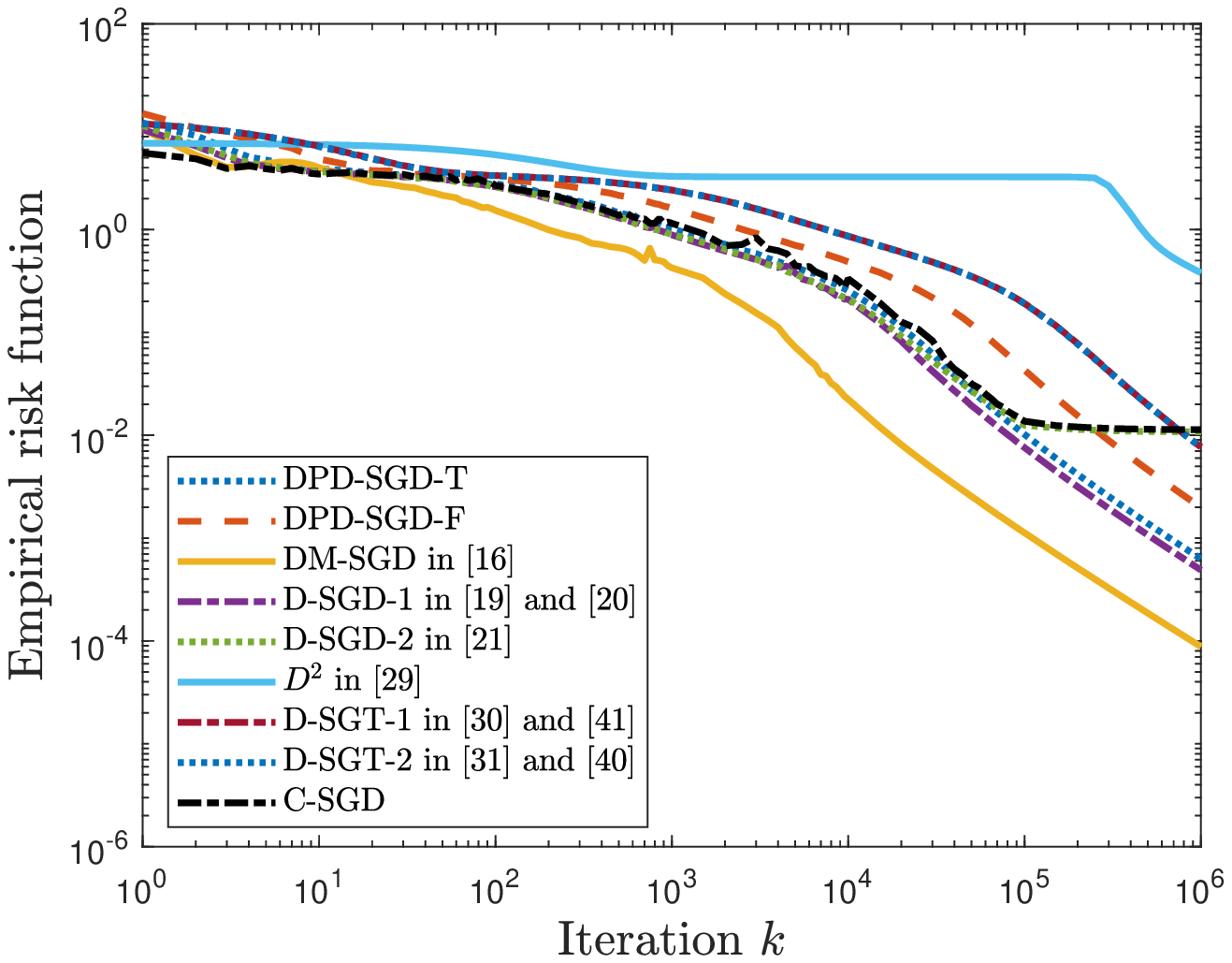}
  \caption{Empirical Risk.}
  \label{sgd:fig:nn_risk}
\end{figure}

Fig.~\ref{sgd:fig:nn_risk} shows that the proposed distributed primal--dual SGD algorithm with time-varying parameters converges almost as fast as the distributed SGD algorithm in \cite{jiang2017collaborative,lian2017can} and faster than the distributed SGD algorithms in \cite{george2019distributed,Tang2018Decentralized,lu2019gnsd,zhang2019decentralized,
pu2018distributed,xin2019distributed} and the centralized SGD algorithm. Note that our algorithm converges slower than the distributed momentum SGD algorithm \cite{Yu2019on}. This is reasonable since that algorithm is an accelerated algorithm with extra requirement on the cost functions, i.e., the deviations between the gradients of local cost functions is bounded, and it requires each agent to communicate three $p$-dimensional variables with its neighbors at each iteration. The slope of the curves are however almost the same. The accuracy of each algorithm is given in Table~\ref{tab:nn-acc}.

\begin{table*}[htbp]
\caption{{Accuracy on each algorithm in NN experiment.} }
\label{tab:nn-acc}
\vskip 0.15in
\begin{center}
\begin{normalsize}
\begin{tabular}{M{4.1cm}|M{3.2cm}}
\hline
Algorithm & Accuracy  \\

\hline
DPD-SGD-T & $93.04\%$\\

\hline
DPD-SGD-F & $92.76\%$\\

\hline
DM-SGD \cite{Yu2019on} & $93.44\%$\\

\hline
D-SGD-1 \cite{jiang2017collaborative,lian2017can}& $92.96\%$\\

\hline
D-SGD-2 \cite{george2019distributed} & $92.88\%$\\

\hline
$D^{2}$ \cite{Tang2018Decentralized}& $90.44\%$\\

\hline
D-SGT-1 \cite{lu2019gnsd,xin2019distributed} & $92.88\%$\\

\hline
D-SGT-2 \cite{zhang2019decentralized,pu2018distributed} & $92.96\%$\\

\hline
C-SGD & $93\%$\\
\hline
\end{tabular}
\end{normalsize}
\end{center}
\vskip -0.1in
\end{table*}

\subsection{Convolutional Neural Networks}
Let us consider the training of a convolutional neural networks (CNN) model.
We build a CNN model for each agent  with five 3$\times$3 convolutional layers using ReLU as activation function, one average pooling layer with filters of size 2$\times$2, one sigmoid layer with dimension 360, another sigmoid layer with dimension 60, one softmax layer with dimension 10. In this experiment, we use the whole MNIST data set. We use the same  communication graph as in above NN  experiment. Each agent is assigned 6000 data points randomly. We set the batch size as 20, which means at each iteration, 20 data points are chosen by the agent to update the gradient, which is also following a uniform distribution. For each algorithm, we do 10 epochs to train the CNN model.

We compare our algorithms DPD-SGD-T and DPD-SGD-F  with the fastest one above: DM-SGD, D-SGD-1, and C-SGD. We list all the parameters we choose in the CNN experiment for each algorithm in Table~\ref{tab:cnn-par}.

\begin{table*}[ht!]
\caption{{Parameters in each algorithm in CNN experiment.} }
\label{tab:cnn-par}
\vskip 0.15in
\begin{center}
\begin{normalsize}
\begin{tabular}{M{3.5cm}|M{2.5cm}|M{2.5cm}|M{2.5cm}}
\hline
Algorithm & $\eta_k$ & $\alpha_k$ & $\beta_k$ \\

\hline
DPD-SGD-T & $0.5/{k^{10^{-5}}}$ & $0.5k^{10^{-5}}$ & $0.1k^{10^{-5}}$\\

\hline
DPD-SGD-F &0.5 & 0.5 & 0.1\\

\hline
DM-SGD \cite{Yu2019on} & 0.1 & \ding{55} &0.8\\

\hline
D-SGD \cite{jiang2017collaborative,lian2017can}& 0.1 &\ding{55} &\ding{55}\\

\hline
C-SGD & 0.1 &\ding{55} &\ding{55}\\
\hline
\end{tabular}
\end{normalsize}
\end{center}
\vskip -0.1in
\end{table*}

We demonstrate the training loss and the test accuracy of each algorithm in Fig.~\ref{sgd:fig:cnn_loss} and Fig.~\ref{sgd:fig:cnn_acc}. Here we use Categorical Cross-Entropy loss, which is a softmax activation plus a Cross-Entropy loss. We can see that our algorithms perform almost the same as the DM-SGD and better than the D-SGD-1 and  the centralized C-SGD. The accuracy of each algorithm is given in Table~\ref{tab:cnn-acc}.

\begin{figure}[!ht]
\centering
        \includegraphics[width=0.85\textwidth]{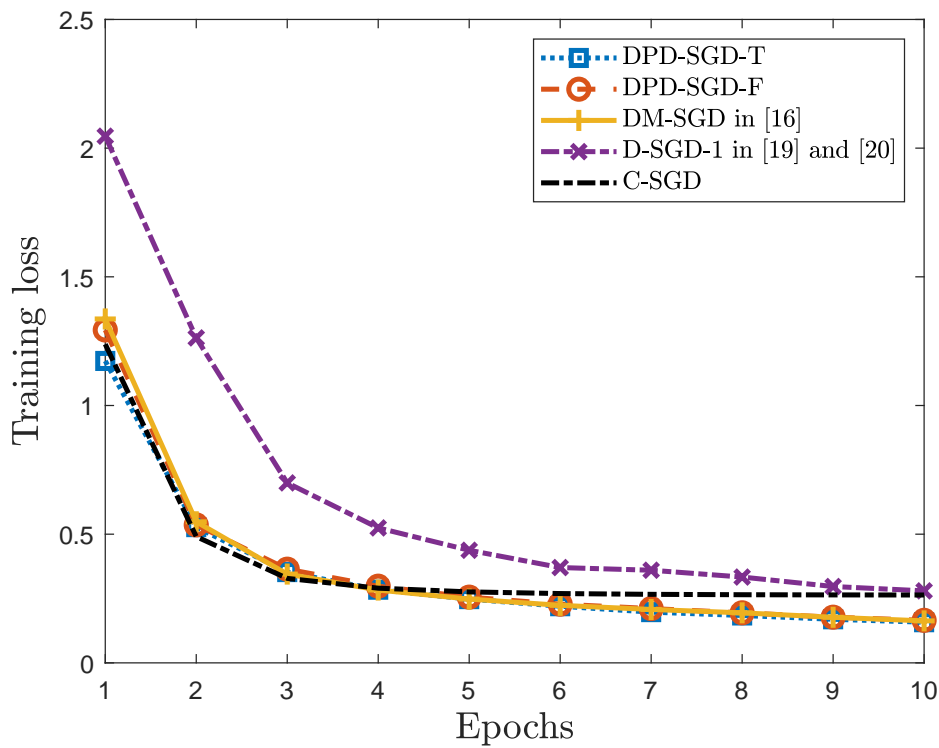}
        \caption{CNN training loss.}
        \label{sgd:fig:cnn_loss}
\end{figure}

\begin{figure}[!ht]
\centering
        \includegraphics[width=0.85\textwidth]{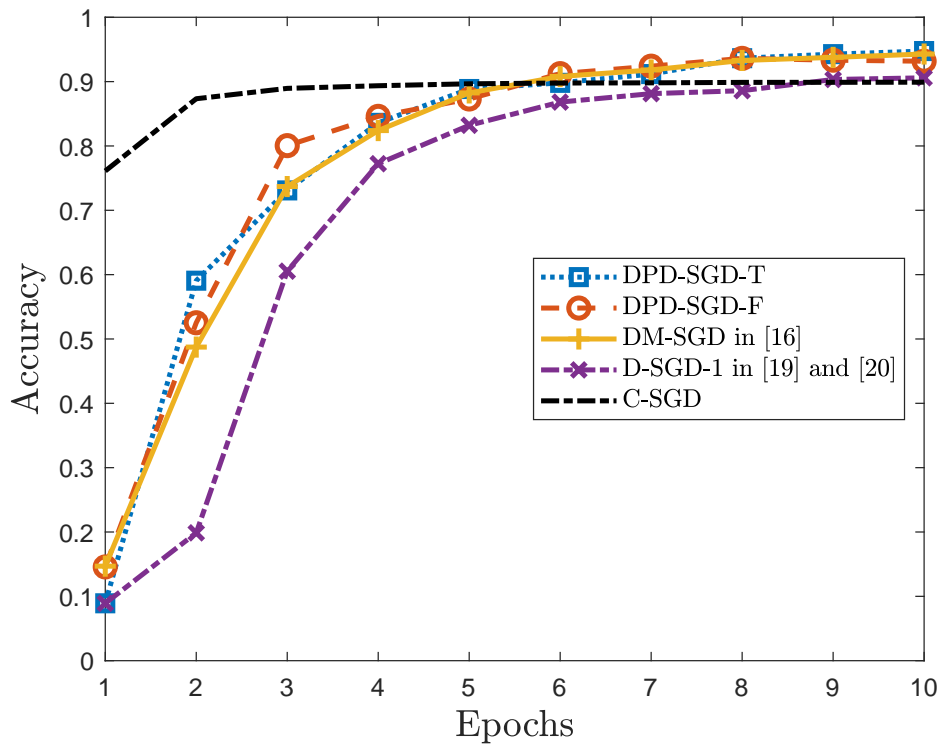}
        \caption{CNN accuracy.}
        \label{sgd:fig:cnn_acc}
\end{figure}

\begin{table*}[ht!]
\caption{{Accuracy on each algorithm in CNN experiment.} }
\label{tab:cnn-acc}
\vskip 0.15in
\begin{center}
\begin{normalsize}
\begin{tabular}{M{3.5cm}|M{2.5cm}}
\hline
Algorithm & Accuracy  \\

\hline
DPD-SGD-T & $94.75\%$\\

\hline
DPD-SGD-F & $93.17\%$\\

\hline
DM-SGD \cite{Yu2019on} & $94.29\%$\\

\hline
D-SGD \cite{jiang2017collaborative,lian2017can}& $92.96\%$\\

\hline
C-SGD & $89.91\%$\\
\hline
\end{tabular}
\end{normalsize}
\end{center}
\vskip -0.1in
\end{table*}

\section{Conclusions}\label{sgd:sec-conclusion}
In this paper, we studied distributed nonconvex optimization. We proposed a distributed primal--dual SGD algorithm and derived its convergence rate. More specifically, the linear speedup convergence rate $\mathcal{O}(1/\sqrt{nT})$ was established for smooth nonconvex cost functions under  arbitrarily connected communication networks.  The convergence rate was improved to the linear speedup convergence rate $\mathcal{O}(1/(nT))$ when the global cost function additionally satisfies the P--{\L} condition. It was also shown that the output of the proposed algorithm with constant parameters linearly converges to a neighborhood of a global optimum. Interesting directions for future work include achieving linear speedup under the P--{\L} condition while considering communication reduction. 

\bibliographystyle{IEEEtran}
\bibliography{refextra}









\appendix

\subsection{Notations and Useful Lemmas}\label{sgd:app-lemmas}
${\bm I}_n$ is the $n$-dimensional identity matrix. The notation $A\otimes B$ denotes the Kronecker product
of matrices $A$ and $B$. $\nullrank(A)$ is the null space of matrix $A$.
Given two symmetric matrices $M,N$, $M\ge N$ means that $M-N$ is positive semi-definite. $\rho(\cdot)$ stands for the spectral radius for matrices and $\rho_2(\cdot)$ indicates the minimum
positive eigenvalue for matrices having positive eigenvalues. For any square matrix $A$, $\|x\|_A^2$ denotes $x^\top Ax$. $\lceil \cdot\rceil$ and $\lfloor\cdot\rfloor$ denote the ceiling and floor functions, respectively. For any $x\in\mathbb{R}$, $[x]_+$ is the positive part of $x$. ${\bm 1}_{(\cdot)}$ is the indicator function. For any $n\in\mathbb{N}_0$, $n!$ is the factorial of $n$.

Denote $K_n={\bm I}_n-\frac{1}{n}{\bm 1}_n{\bm 1}^{\top}_n$, $\bsK=K_n\otimes {\bm I}_p$, $\bsH=\frac{1}{n}({\bm 1}_n{\bm 1}_n^\top\otimes{\bm I}_p)$, $\bar{x}_k=\frac{1}{n}({\bm 1}_n^\top\otimes{\bm I}_p)\bsx_k$, $\bar{\bsx}_k={\bm 1}_n\otimes\bar{x}_k$, $\bsg_k=\nabla\tilde{f}(\bsx_k)$, $\bar{\bsg}_k=\bsH\bsg_{k}$, $\bsg^0_k=\nabla\tilde{f}(\bar{\bsx}_k)$, $\bar{\bsg}_k^0=\bsH\bsg^0_{k}={\bm 1}_n\otimes\nabla f(\bar{x}_k)$, and $\bar{\bsg}_k^u=\bsH\bsg_k^u$.

The following results are used in the proofs.

\begin{lemma} (Lemma~1.2.3 in \cite{nesterov2018lectures} and Lemma~3 in \cite{tang2020distributedzero})
If the function $f(x):~\mathbb{R}^p\mapsto\mathbb{R}$ is smooth with constant $L_f>0$, then
\begin{subequations}
\begin{align}
&|f(y)-f(x)-(y-x)^\top\nabla f(x)|
\le\frac{L_f}{2}\|y-x\|^2, \label{nonconvex:lemma:lipschitz}\\
&\|\nabla f(x)\|^2\le2L_f(f(x)-f^*),~\forall x,y\in\mathbb{R}^{p},\label{nonconvex:lemma:lipschitz2}
\end{align}
\end{subequations}
where $f^*=\min_{x\in\mathbb{R}^p}f(x)$.
\end{lemma}

\begin{lemma}\label{nonconvex:lemma-Xinlei} (Lemmas~1 and 2 in \cite{Yi2018distributed})
Let $L$ be the Laplacian matrix of the graph $\mathcal{G}$.
If Assumption~\ref{sgd:assgraph} holds, then $L$ is positive semi-definite, $\nullrank(L)=\nullrank(K_n)=\{{\bm 1}_n\}$, $L\le\rho(L){\bm I}_n$, $\rho(K_n)=1$,
\begin{align}
&K_nL=LK_n=L,\label{nonconvex:KL-L-eq}\\
&0\le\rho_2(L)K_n\le L\le\rho(L)K_n.\label{nonconvex:KL-L-eq2}
\end{align}
Moreover, there exists an orthogonal matrix $[r \ R]\in \mathbb{R}^{n \times n}$ with $r=\frac{1}{\sqrt{n}}\mathbf{1}_n$ and $R \in \mathbb{R}^{n\times (n-1)}$ such that
\begin{align}
&R\Lambda_1^{-1}R^{\top}L=LR\Lambda_1^{-1}R^{\top}=K_n,\label{nonconvex:lemma-eq}\\
&\frac{1}{\rho(L)}K_n\leq R\Lambda_1^{-1}R^{\top}\le\frac{1}{\rho_2(L)}K_n,\label{nonconvex:lemma-eq2}
\end{align}
where $\Lambda_1=\diag([\lambda_2,\dots,\lambda_n])$ with $0<\lambda_2\leq\dots\leq\lambda_n$ being the eigenvalues of the Laplacian matrix $L$.
\end{lemma}

\begin{lemma}\label{sgu:lemma:exponential}
Let $a\in(0,1)$ be a constant, then
\begin{align}\label{sgu:lemma:exponential-equ}
(1-a)^T\le\frac{k!}{(aT)^k},~\forall k,T\in\mathbb{N}_0.
\end{align}
\end{lemma}
\begin{proof}
For any constant $a\in(0,1)$, we have $\ln(1-a)\le-a$. Thus,
\begin{align}\label{sgu:lemma:exponential-equ1}
(1-a)^T\le e^{-aT},~\forall T\in\mathbb{N}_0.
\end{align}

For any constant $x>0$, we have $e^x>\frac{x^k}{k!},~\forall k\in\mathbb{N}_0$. This result together with \eqref{sgu:lemma:exponential-equ1} yields  \eqref{sgu:lemma:exponential-equ}.
\end{proof}

\begin{lemma}\label{zerosg:serise:lemma:sequence}
Let $\{z_k\}$, $\{r_{1,k}\}$, and $\{r_{2,k}\}$ be sequences. Suppose there exists $t_1\in\mathbb{N}_+$ such that
\begin{subequations}
\begin{align}
&z_k\ge0,\label{zerosg:serise:lemma:sequence-equ0}\\
&z_{k+1}\le(1-r_{1,k})z_k+r_{2,k},\label{zerosg:serise:lemma:sequence-equ1}\\
&1> r_{1,k}\ge\frac{a_1}{(k+t_1)^\delta},\label{zerosg:serise:lemma:sequence-equ2}\\
&r_{2,k}\le\frac{a_2}{(k+t_1)^2},~\forall k\in\mathbb{N}_0, \label{zerosg:serise:lemma:sequence-equ3}
\end{align}
\end{subequations}
where $\delta\ge0$, $a_1>0$, and $a_2>0$ are constants.

(i) If $\delta=1$, then
\begin{align}\label{zerosg:serise:lemma:sequence-equ5}
z_{k}&\le \phi_1(k,t_1,a_1,a_2,z_{0}),~\forall k\in\mathbb{N}_+,
\end{align}
where
\begin{align}\label{zerosg:serise:lemma:sequence-equ5-phi3}
\phi_1(k,t_1,a_1,a_2,z_{0})&=\frac{t_1^{a_1}z_{0}}{(k+t_1)^{a_1}}
+\frac{a_2}{(k+t_1-1)^{2}}\nonumber\\
&\quad+4a_2s_1(k+t_1),
\end{align}
and
\begin{align*}
s_1(k)=
\begin{cases}
  \frac{1}{(a_1-1)k}, & \mbox{if } a_1>1, \\
  \frac{\ln(k-1)}{k}, & \mbox{if } a_1=1, \\
  \frac{-t_1^{a_1-1}}{(a_1-1)k^{a_1}}, & \mbox{if } a_1<1.
\end{cases}
\end{align*}

(ii) If $\delta=0$, then
\begin{align}\label{zerosg:serise:lemma:sequence-equ6}
z_{k}&\le \phi_2(k,t_1,a_1,a_2,z_{0}),~\forall k\in\mathbb{N}_+,
\end{align}
where
\begin{align}\label{zerosg:serise:lemma:sequence-equ6-phi4}
\phi_2(k,t_1,a_1,a_2,z_{0})
&=(1-a_1)^kz_{0}+a_2(1-a_1)^{k+t_1-1}\Big([t_2-t_1]_+s_2(t_1)\nonumber\\
&\quad+([t_3-t_1]_+-[t_2-t_1]_+)s_2(t_3)\Big)\nonumber\\
&\quad+\frac{{\bm 1}_{(k+t_1-1\ge t_3)}2a_2}{-\ln(1-a_1)(k+t_1)^{2}(1-a_1)},
\end{align}
$s_2(k)=\frac{1}{k^{2}(1-a_1)^{k}}$, $t_2=\lceil \frac{-2}{\ln(1-a_1)}\rceil$, and $t_3=\lceil \frac{-4}{\ln(1-a_1)}\rceil$.
\end{lemma}
\begin{proof}
(i) From \eqref{zerosg:serise:lemma:sequence-equ0}--\eqref{zerosg:serise:lemma:sequence-equ2}, for any $k\in\mathbb{N}_+$, it holds that
\begin{align}\label{zerosg:serise:lemma:seqproof1}
z_{k}&\le \prod_{\tau=0}^{k-1}(1-r_{1,\tau})z_{0}+r_{2,k-1}
+\sum_{l=0}^{k-2}\prod_{\tau=l+1}^{k-1}(1-r_{1,\tau})r_{2,l}.
\end{align}


For any $t\in[0, 1]$, it holds that $1-t\le e^{-t}$ since $s_3(t)=1-t-e^{-t}$ is a non-increasing function in the interval $[0, 1]$. Thus, for any $k>l\ge 0$, it holds that
\begin{align}\label{zerosg:serise:lemma:seqproof2}
\prod_{\tau=l}^{k-1}(1-r_{1,\tau})
\le e^{-\sum_{\tau=l}^{k-1}r_{1,\tau}}.
\end{align}

We also have
\begin{align}\label{zerosg:serise:lemma:seqproof3}
&\sum_{\tau=l}^{k-1}r_{1,\tau}\ge\sum_{\tau=l}^{k-1}\frac{a_1}{\tau+t_1}
=\sum_{\tau=l+t_1}^{k-1+t_1}\frac{a_1}{\tau}
\ge\int_{t=l+t_1}^{k+t_1}\frac{a_1}{t}dt=a_1(\ln(k+t_1)-\ln(l+t_1)),
\end{align}
where the first inequality holds due to \eqref{zerosg:serise:lemma:sequence-equ2} and the second inequality holds since $s_4(t)=a_1/t$ is a decreasing function in the interval $[1, +\infty)$.

Hence, \eqref{zerosg:serise:lemma:seqproof2} and \eqref{zerosg:serise:lemma:seqproof3} yield
\begin{align}\label{zerosg:serise:lemma:seqproof4}
&\prod_{\tau=l}^{k-1}(1-r_{1,\tau})
\le e^{-\sum_{\tau=l}^{k-1}r_{1,\tau}}\le\frac{(l+t_1)^{a_1}}{(k+t_1)^{a_1}}.
\end{align}

We have
\begin{align}\label{zerosg:serise:lemma:seqproof9}
\sum_{l=0}^{k-2}\prod_{\tau=l+1}^{k-1}(1-r_{1,\tau})r_{2,l}
&\le\sum_{l=0}^{k-2}\frac{(l+t_1+1)^{a_1}}{(k+t_1)^{a_1}}
\frac{a_2}{(l+t_1)^2}\nonumber\\
&\le\sum_{l=0}^{k-2}\frac{(l+t_1+1)^{a_1}}{(k+t_1)^{a_1}}
\frac{a_2}{(\frac{t_1}{t_1+1}l+t_1)^2}\nonumber\\
&=\frac{(\frac{t_1+1}{t_1})^2a_2}{(k+t_1)^{a_1}}
\sum_{l=0}^{k-2}\frac{(l+t_1+1)^{a_1}}{(l+t_1+1)^2}\nonumber\\
&=\frac{4a_2}{(k+t_1)^{a_1}}
\sum_{l=t_1+1}^{k+t_1-1}l^{a_1-2},
\end{align}
where the first inequality holds due to \eqref{zerosg:serise:lemma:seqproof4} and \eqref{zerosg:serise:lemma:sequence-equ3}.

From \eqref{zerosg:serise:lemma:seqproof1}, \eqref{zerosg:serise:lemma:seqproof4}, and \eqref{zerosg:serise:lemma:seqproof9}, we have \eqref{zerosg:serise:lemma:sequence-equ5}.

(ii) Denote $a=1-a_1$.  From \eqref{zerosg:serise:lemma:sequence-equ2} and $\delta=0$, we know that $a_1\in(0,1)$. Thus, $a\in(0,1)$.

From \eqref{zerosg:serise:lemma:sequence-equ0}--\eqref{zerosg:serise:lemma:sequence-equ3} and $\delta_1=0$, for any $k\in\mathbb{N}_+$, it holds that
\begin{align}\label{zerosg:serise:lemma:seqproof1-0}
z_{k}&\le (1-a_1)^kz_{0}+\sum_{\tau=0}^{k-1}(1-a_1)^{k-1-\tau}r_{2,\tau}
\le a^kz_{0}+a_2a^{k+t_1-1}\sum_{\tau=0}^{k-1}\frac{1}{(\tau+t_1)^{2}a^{\tau+t_1}}.
\end{align}

We have
\begin{align}\label{zerosg:serise:lemma:seqproof1.1-0}
&\sum_{\tau=0}^{k-1}\frac{1}{(\tau+t_1)^{2}a^{\tau+t_1}}
=\sum_{\tau=t_1}^{k+t_1-1}\frac{1}{\tau^{2}a^{\tau}}
=\sum_{\tau=t_1}^{t_2-1}s_2(\tau)
+\sum_{\tau=t_2}^{t_3-1}s_2(\tau)
+\sum_{\tau=t_3}^{k+t_1-1}s_2(\tau).
\end{align}

We know that $s_2(t)=1/(t^{2}a^{t})$ is decreasing and increasing in the intervals $[1,t_2-1]$ and $[t_2,+\infty)$, respectively, since
\begin{align*}
\frac{ds_2(t)}{dt}&=-s_2(t)\Big(\frac{2}{t}+\ln(a)\Big)\le0,~\forall t\in\Big(0,\frac{-2}{\ln(a)}\Big],\\
\frac{ds_2(t)}{dt}&=-s_2(t)\Big(\frac{2}{t}+\ln(a)\Big)\ge0,~\forall t\in\Big[\frac{-2}{\ln(a)},+\infty\Big).
\end{align*}
Thus, we have
\begin{subequations}
\begin{align}
&\sum_{\tau=k_1}^{t_2-1}s_2(\tau)
\le(t_2-k_1)s_2(k_1),~\forall k_1\in[1,t_2-1],\label{zerosg:serise:lemma:seqproof1.2-0}\\
&\sum_{\tau=k_2}^{t_3-1}s_2(\tau)
\le(t_3-k_2)s_2(t_3),~\forall k_2\in[t_2,t_3-1],\label{zerosg:serise:lemma:seqproof1.3-0}\\
&\sum_{\tau=t_3}^{k_3}s_2(\tau)
\le\int_{t_3}^{k_3+1}s_2(t)dt,~\forall k_3\ge t_3.\label{zerosg:serise:lemma:seqproof1.4-0}
\end{align}
\end{subequations}

Denote $b=1/a$. We have
\begin{align}\label{zerosg:serise:lemma:seqproof1.5-0}
\int_{t_3}^{k_3+1}s_2(t)dt&=\int_{t_3}^{k_3+1}\frac{b^t}{t^{2}}dt
=\int_{t_3}^{k_3+1}\frac{1}{\ln(b) t^{2}}db^{t}\nonumber\\
&=\frac{b^{k_3+1}}{\ln(b) (k_3+1)^{2}}-\frac{b^{t_3}}{\ln(b) t_3^{2}}
+\int_{t_3}^{k_3+1}\frac{2b^{t}}{\ln(b) t^{3}}dt\nonumber\\
&\le\frac{b^{k_3+1}}{\ln(b) (k_3+1)^{2}}
+\int_{t_3}^{k_3+1}\frac{2}{\ln(b) t}s_2(t)dt\nonumber\\
&\le\frac{b^{k_3+1}}{\ln(b) (k_3+1)^{2}}
+\frac{2}{\ln(b) t_3}\int_{t_3}^{k_3+1}s_2(t)dt\nonumber\\
&\le\frac{b^{k_3+1}}{\ln(b) (k_3+1)^{2}}
+\frac{1}{2}\int_{t_3}^{k_3+1}s_2(t)dt,
\end{align}
where the last inequality holds due to $t_3=\lceil \frac{-4}{\ln(1-a_1)}\rceil\ge\frac{-4}{\ln(1-a_1)}
=\frac{4}{\ln(b)}$.

From \eqref{zerosg:serise:lemma:seqproof1.4-0} and \eqref{zerosg:serise:lemma:seqproof1.5-0}, we have
\begin{align}
&\sum_{\tau=t_3}^{k_3}s_2(\tau)
\le\frac{-2}{\ln(a) (k_3+1)^{2}a^{k_3+1}},~\forall k_3\ge t_3.\label{zerosg:serise:lemma:seqproof1.6-0}
\end{align}

From \eqref{zerosg:serise:lemma:seqproof1-0}, \eqref{zerosg:serise:lemma:seqproof1.1-0}, \eqref{zerosg:serise:lemma:seqproof1.2-0}, \eqref{zerosg:serise:lemma:seqproof1.3-0}, and \eqref{zerosg:serise:lemma:seqproof1.6-0}, we get \eqref{zerosg:serise:lemma:sequence-equ6}.
\end{proof}

\begin{lemma}
Suppose Assumptions~\ref{sgd:assgraph} and \ref{sgd:assfiu}--\ref{sgd:ass:stochastic-grad:variance} hold. Then the following holds for Algorithm~\ref{sguT:algorithm-sg}
\begin{align}
\mathbf{E}_{\mathfrak{F}_k}[W_{1,k+1}]
&\le W_{1,k}-\|\bsx_k\|^2_{\eta_k\alpha_k\bsL-\frac{1}{2}\eta_k\bsK
-\frac{3}{2}\eta^2_k\alpha^2_k\bsL^2-\frac{1}{2}\eta_k(1+5\eta_k)L_f^2\bsK}
\nonumber\\
&\quad-\eta_k\beta_k\bsx^\top_k\bsK\Big(\bm{v}_k+\frac{1}{\beta_k}\bsg_k^0\Big)
+\Big\|\bm{v}_k+\frac{1}{\beta_k}\bsg_k^0\Big\|^2_{\frac{3}{2}\eta^2_k\beta^2_k\bsK}
+2n\sigma^2\eta^2_k,\label{sguT:v1k}
\end{align}
where $W_{1,k}=\frac{1}{2}\|\bm{x}_k \|^2_{\bsK}$.
\end{lemma}
\begin{proof}
Noting that $\nabla\tilde{f}$ is Lipschitz-continuous with constant $L_{f}>0$ since Assumption~\ref{sgd:assfiu} is satisfied, we have that
\begin{align}
&\|\bsg^0_{k}-\bsg_{k}\|^2\le L_f^2\|\bar{\bsx}_{k}-\bsx_{k}\|^2=L_f^2\|\bsx_{k}\|^2_{\bsK}.\label{sguT:gg1}
\end{align}

From Assumptions~\ref{sgd:ass:stochastic-grad:xi}--\ref{sgd:ass:stochastic-grad:variance}, we know that
\begin{subequations}
\begin{align}
&\mathbf{E}_{\mathfrak{F}_k}[\bsg^u_k]=\bsg_k,\label{sguT:sgproof-hg1}\\
&\mathbf{E}_{\mathfrak{F}_k}[\|\bsg^u_k-\bsg_k\|^2]\le n\sigma^2.\label{sguT:sgproof-hg2}
\end{align}
\end{subequations}

From \eqref{sguT:gg1}, \eqref{sguT:sgproof-hg2}, and the Cauchy--Schwarz inequality, we have
\begin{align}
\mathbf{E}_{\mathfrak{F}_k}[\|\bsg_k^0-\bsg^u_k\|^2]
&=\mathbf{E}_{\mathfrak{F}_k}[\|\bsg_k^0-\bsg_k+\bsg_k-\bsg^u_k\|^2]\nonumber\\
&\le2\|\bsg_k^0-\bsg_k\|^2+2\mathbf{E}_{\mathfrak{U}_k}[\|\bsg_k-\bsg^u_k\|^2]\nonumber\\
&\le2L_f^2\|\bsx_{k}\|^2_{\bsK}+2 n\sigma^2.\label{sguT:sgproof-hg5}
\end{align}

We have
\begin{align}
\mathbf{E}_{\mathfrak{F}_k}[W_{1,k+1}]
&=\mathbf{E}_{\mathfrak{F}_k}\Big[\frac{1}{2}\|\bm{x}_{k+1} \|^2_{\bsK}\Big]\nonumber\\
&=\mathbf{E}_{\mathfrak{F}_k}\Big[\frac{1}{2}\|\bm{x}_k-\eta_k(\alpha_k\bsL\bm{x}_k+\beta_k\bm{v}_k+\bsg^u_k) \|^2_{\bsK}\Big]\nonumber\\
&=\mathbf{E}_{\mathfrak{F}_k}\Big[\frac{1}{2}\|\bm{x}_k\|^2_{\bsK}-\eta_k\alpha_k\|\bsx_k\|^2_{\bsL}
+\frac{1}{2}\eta^2_k\alpha^2_k\|\bsx_k\|^2_{\bsL^2}
\nonumber\\
&\quad-\eta_k\beta_k\bsx^\top_k({\bm I}_{np}-\eta_k\alpha_k\bsL)\bsK\Big(\bm{v}_k+\frac{1}{\beta_k}\bsg^u_k\Big)
+\frac{1}{2}\eta^2_k\beta^2_k\Big\|\bm{v}_k+\frac{1}{\beta_k}\bsg^u_k\Big\|^2_{\bsK}\Big]\nonumber\\
&=\frac{1}{2}\|\bm{x}_k\|^2_{\bsK}-\|\bsx_k\|^2_{\eta_k\alpha_k\bsL
-\frac{1}{2}\eta^2_k\alpha^2_k\bsL^2}-\eta_k\beta_k\bsx^\top_k({\bm I}_{np}\nonumber\\
&\quad-\eta_k\alpha_k\bsL)\bsK\Big(\bm{v}_k
+\frac{1}{\beta_k}\bsg_k^0
+\frac{1}{\beta_k}\bsg_k-\frac{1}{\beta_k}\bsg_k^0\Big)\nonumber\\
&\quad+\frac{1}{2}\eta^2_k\beta^2_k\mathbf{E}_{\mathfrak{F}_k}
\Big[\Big\|\bm{v}_k+\frac{1}{\beta_k}\bsg_k^0
+\frac{1}{\beta_k}\bsg^u_k-\frac{1}{\beta_k}\bsg_k^0\Big\|^2_{\bsK}\Big]\nonumber\\
&\le W_{1,k}-\|\bsx_k\|^2_{\eta_k\alpha_k\bsL
-\frac{1}{2}\eta^2_k\alpha^2_k\bsL^2}
-\eta_k\beta_k\bsx^\top_k\bsK\Big(\bm{v}_k+\frac{1}{\beta_k}\bsg_k^0\Big)\nonumber\\
&\quad+\frac{\eta_k}{2}\|\bm{x}_k\|^2_{\bsK}
+\frac{\eta_k}{2}\|\bsg_k-\bsg_k^0\|^2
+\frac{1}{2}\eta^2_k\alpha^2_k\|\bm{x}_k\|^2_{\bsL^2}
+\frac{1}{2}\eta^2_k\beta^2_k\Big\|\bm{v}_k+\frac{1}{\beta_k}\bsg_k^0\Big\|^2_{\bsK}\nonumber\\
&\quad+\frac{1}{2}\eta^2_k\alpha^2_k\|\bm{x}_k\|^2_{\bsL^2}
+\frac{1}{2}\eta^2_k\|\bsg_k-\bsg_k^0\|^2\nonumber\\
&\quad+\eta^2_k\beta^2_k\Big\|\bm{v}_k+\frac{1}{\beta_k}\bsg_k^0\Big\|^2_{\bsK}
+\eta^2_k\mathbf{E}_{\mathfrak{F}_k}[\|\bsg^u_k-\bsg_k^0\|^2]\nonumber\\
&=W_{1,k}-\|\bsx_k\|^2_{\eta_k\alpha_k\bsL-\frac{1}{2}\eta_k\bsK
-\frac{3}{2}\eta^2_k\alpha^2_k\bsL^2}\nonumber\\
&\quad+\frac{\eta_k}{2}(1+\eta_k)\|\bsg_k-\bsg_k^0\|^2
+\eta^2_k\mathbf{E}_{\mathfrak{F}_k}[\|\bsg^u_k-\bsg_k^0\|^2]
\nonumber\\
&\quad-\eta_k\beta_k\bsx^\top_k\bsK\Big(\bm{v}_k+\frac{1}{\beta_k}\bsg_k^0\Big)
+\Big\|\bm{v}_k+\frac{1}{\beta_k}\bsg_k^0\Big\|^2_{\frac{3}{2}\eta^2_k\beta^2_k\bsK},
\label{sguT:v1k-1}
\end{align}
where the second equality holds due to \eqref{sgu:kia-algo-dc-compact-x}; the third equality holds due to \eqref{nonconvex:KL-L-eq} in Lemma~\ref{nonconvex:lemma-Xinlei}; the fourth equality holds since $\bsx_{k}$ and $\bsv_{k}$ are independent of $\mathfrak{F}_k$ and \eqref{sguT:sgproof-hg1}; and the inequality holds due to the Cauchy--Schwarz inequality and $\rho(\bsK)=1$.

Then, from \eqref{sguT:gg1}, \eqref{sguT:sgproof-hg5}, and \eqref{sguT:v1k-1}, we have \eqref{sguT:v1k}.
\end{proof}

\begin{lemma}
Suppose Assumptions~\ref{sgd:assgraph} and \ref{sgd:assfiu} hold, and $\{\beta_k\}$ is non-decreasing. Then the following holds for Algorithm~\ref{sguT:algorithm-sg}
\begin{align}\label{sguT:v2k}
W_{2,k+1}
&\le W_{2,k}
+(1+\omega_k)\eta_k\beta_k\bsx^\top_k(\bsK+\kappa_1\bsL)\Big(\bm{v}_k+\frac{1}{\beta_k}\bsg_k^0\Big)\nonumber\\
&\quad+\frac{1}{2}(\eta_k+\omega_k+\eta_k\omega_k)\Big(\frac{1}{\rho_2(L)}+\kappa_1\Big)
\Big\|\bm{v}_k+\frac{1}{\beta_k}\bsg_{k}^0\Big\|^2_{\bsK}\nonumber\\
&\quad+\|\bsx_k\|^2_{(1+\omega_k)\eta^2_k\beta^2_k(\bsL+\kappa_1\bsL^2)}
+\frac{\eta_k}{\beta^2_k}\Big(\eta_k+\frac{1}{2}\Big)(1+\omega_k)
\Big(\frac{1}{\rho_2(L)}+\kappa_1\Big)L_f^2\|\bar{\bsg}^u_k\|^2\nonumber\\
&\quad+\frac{1}{2}\Big(\frac{1}{\rho_2(L)}+\kappa_1\Big)(\omega_k+\omega_k^2)\|\bsg_{k+1}^0\|^2,
\end{align}
where $W_{2,k}=\frac{1}{2}\|\bsv_k+\frac{1}{\beta_k}\bsg_k^0\|^2_{\bsQ+\kappa_1\bsK}$, $\bsQ=R\Lambda^{-1}_1R^{\top}\otimes {\bm I}_p$ with matrices $R$ and $\Lambda^{-1}_1$  given in Lemma~\ref{nonconvex:lemma-Xinlei}, $\omega_k=\frac{1}{\beta_{k}}-\frac{1}{\beta_{k+1}}$, and $\kappa_1>0$ is a constant.
\end{lemma}
\begin{proof}
Denote $\bar{v}_k=\frac{1}{n}({\bm 1}_n^\top\otimes{\bm I}_p)\bsv_k$. Then,
from \eqref{sgu:kia-algo-dc-compact-v}, we know that
\begin{align}
\bar{v}_{k+1}=\bar{v}_k.\label{sguT:vbardynamic}
\end{align}
Then, from \eqref{sguT:vbardynamic} and $\sum_{i=1}^{n}v_{i,0}={\bm 0}_p$, we know that
\begin{align}
\bar{v}_k={\bm 0}_p.\label{sguT:vkn}
\end{align}
Then, from \eqref{sguT:vkn} and \eqref{sgu:kia-algo-dc-compact-x}, we know that
\begin{align}
\bar{\bsx}_{k+1}=\bar{\bsx}_{k}-\eta_k\bar{\bsg}^u_k.\label{sguT:xbardynamic-sg}
\end{align}

Since $\nabla\tilde{f}$ is Lipschitz-continuous and \eqref{sguT:xbardynamic-sg}, we have
\begin{align}
&\|\bsg^0_{k+1}-\bsg^0_{k}\|^2\le L_f^2\|\bar{\bsx}_{k+1}-\bar{\bsx}_{k}\|^2=\eta^2_kL_f^2\|\bar{\bsg}^u_k\|^2.\label{sguT:gg}
\end{align}

We know that $\omega_k\ge0$ since $\{\beta_k\}$ is non-decreasing. We have
\begin{align}
W_{2,k+1}
&=\frac{1}{2}\Big\|\bsv_{k+1}+\frac{1}{\beta_{k+1}}\bsg_{k+1}^0\Big\|^2_{\bsQ+\kappa_1\bsK}\nonumber\\
&=\frac{1}{2}\Big\|\bsv_{k+1}+\frac{1}{\beta_{k}}\bsg_{k+1}^0
+\Big(\frac{1}{\beta_{k+1}}-\frac{1}{\beta_{k}}\Big)\bsg_{k+1}^0\Big\|^2_{\bsQ+\kappa_1\bsK}\nonumber\\
&\le\frac{1}{2}(1+\omega_k)\Big\|\bsv_{k+1}+\frac{1}{\beta_{k}}\bsg_{k+1}^0\Big\|^2_{\bsQ
+\kappa_1\bsK}
+\frac{1}{2}(\omega_k+\omega_k^2)\|\bsg_{k+1}^0\|^2_{\bsQ+\kappa_1\bsK},\label{sguT:v2k1}
\end{align}
where the inequality holds due to the Cauchy--Schwarz inequality.

For the first term in the right-hand side of \eqref{sguT:v2k1}, we have
\begin{align}
\frac{1}{2}\Big\|\bsv_{k+1}+\frac{1}{\beta_{k}}\bsg_{k+1}^0\Big\|^2_{\bsQ+\kappa_1\bsK}
&=\frac{1}{2}\Big\|\bm{v}_k+\frac{1}{\beta_k}\bsg_{k}^0+\eta_k\beta_k\bsL\bm{x}_k
+\frac{1}{\beta_k}(\bsg_{k+1}^0-\bsg_{k}^0) \Big\|^2_{\bsQ+\kappa_1\bsK}\nonumber\\
&=\frac{1}{2}\Big\|\bm{v}_k+\frac{1}{\beta_k}\bsg_{k}^0\Big\|^2_{\bsQ+\kappa_1\bsK}
+\eta_k\beta_k\bsx^\top_k(\bsK+\kappa_1\bsL)\Big(\bm{v}_k+\frac{1}{\beta_k}\bsg_k^0\Big)\nonumber\\
&\quad+\|\bsx_k\|^2_{\frac{1}{2}\eta_k^2\beta_k^2(\bsL+\kappa_1\bsL^2)}
+\frac{1}{2\beta^2_k}\|\bsg_{k+1}^0-\bsg_{k}^0\|^2_{\bsQ+\kappa_1\bsK}\nonumber\\
&\quad+\frac{1}{\beta_k}\Big(\bm{v}_k+\frac{1}{\beta_k}\bsg_{k}^0
+\eta_k\beta_k\bsL\bm{x}_k\Big)^\top(\bsQ
+\kappa_1\bsK)(\bsg_{k+1}^0-\bsg_{k}^0)\nonumber\\
&\le W_{2,k}+\eta_k\beta_k\bsx^\top_k(\bsK+\kappa_1\bsL)\Big(\bm{v}_k+\frac{1}{\beta_k}\bsg_k^0\Big)\nonumber\\
&\quad+\|\bsx_k\|^2_{\frac{1}{2}\eta_k^2\beta_k^2(\bsL+\kappa_1\bsL^2)}
+\frac{1}{2\beta^2_k}\|\bsg_{k+1}^0-\bsg_{k}^0\|^2_{\bsQ+\kappa_1\bsK}\nonumber\\
&\quad+\frac{\eta_k}{2}\Big\|\bm{v}_k+\frac{1}{\beta_k}\bsg_{k}^0\Big\|^2_{\bsQ+\kappa_1\bsK}
+\frac{1}{2\eta_k\beta_k^2}\|\bsg_{k+1}^0-\bsg_{k}^0\|^2_{\bsQ+\kappa_1\bsK}\nonumber\\
&\quad+\frac{1}{2}\eta^2_k\beta^2_k\|\bsL\bm{x}_k\|^2_{\bsQ+\kappa_1\bsK}
+\frac{1}{2\beta^2_k}\|\bsg_{k+1}^0-\bsg_{k}^0\|^2_{\bsQ+\kappa_1\bsK}\nonumber\\
&=W_{2,k}+\eta_k\beta_k\bsx^\top_k(\bsK+\kappa_1\bsL)\Big(\bm{v}_k+\frac{1}{\beta_k}\bsg_k^0\Big)\nonumber\\
&\quad+\|\bsx_k\|^2_{\eta^2_k\beta^2_k(\bsL+\kappa_1\bsL^2)}
+\Big\|\bm{v}_k+\frac{1}{\beta_k}\bsg_{k}^0\Big\|^2_{\frac{1}{2}\eta_k(\bsQ+\kappa_1\bsK)}\nonumber\\
&\quad+\frac{1}{\beta^2_k}\Big(1+\frac{1}{2\eta_k}\Big)\|\bsg_{k+1}^0-\bsg_{k}^0\|^2_{\bsQ+\kappa_1\bsK}\nonumber\\
&\le W_{2,k}+\eta_k\beta_k\bsx^\top_k(\bsK+\kappa_1\bsL)\Big(\bm{v}_k+\frac{1}{\beta_k}\bsg_k^0\Big)\nonumber\\
&\quad+\|\bsx_k\|^2_{\eta^2_k\beta^2_k(\bsL+\kappa_1\bsL^2)}
+\Big\|\bm{v}_k+\frac{1}{\beta_k}\bsg_{k}^0\Big\|^2_{\frac{1}{2}\eta_k(\bsQ+\kappa_1\bsK)}\nonumber\\
&\quad+\frac{1}{\beta^2_k}\Big(1+\frac{1}{2\eta_k}\Big)\Big(\frac{1}{\rho_2(L)}+\kappa_1\Big)
\|\bsg_{k+1}^0-\bsg_{k}^0\|^2\nonumber\\
&\le W_{2,k}+\eta_k\beta_k\bsx^\top_k(\bsK+\kappa_1\bsL)\Big(\bm{v}_k+\frac{1}{\beta_k}\bsg_k^0\Big)\nonumber\\
&\quad+\|\bsx_k\|^2_{\eta^2_k\beta^2_k(\bsL+\kappa_1\bsL^2)}
+\Big\|\bm{v}_k+\frac{1}{\beta_k}\bsg_{k}^0\Big\|^2_{\frac{1}{2}\eta_k(\bsQ+\kappa_1\bsK)}\nonumber\\
&~~~+\frac{\eta_k}{\beta^2_k}\Big(\eta_k+\frac{1}{2}\Big)
\Big(\frac{1}{\rho_2(L)}+\kappa_1\Big)L_f^2\|\bar{\bsg}^u_k\|^2,\label{sguT:v2k2}
\end{align}
where the first equality holds due to \eqref{sgu:kia-algo-dc-compact-v}; the second equality holds due to \eqref{nonconvex:KL-L-eq} and \eqref{nonconvex:lemma-eq} in Lemma~\ref{nonconvex:lemma-Xinlei};  the first inequality holds due to the Cauchy--Schwarz inequality; the last equality holds due to \eqref{nonconvex:KL-L-eq} and \eqref{nonconvex:lemma-eq} in Lemma~\ref{nonconvex:lemma-Xinlei}; the second inequality holds due to $\rho(\bsQ+\kappa_1\bsK)\le\rho(\bsQ)+\kappa_1\rho(\bsK)$, \eqref{nonconvex:lemma-eq2}, $\rho(\bsK)=1$; and the last  inequality holds due to \eqref{sguT:gg}.

For the second term in the right-hand side of \eqref{sguT:v2k1}, we have
\begin{align}\label{sguT:v2k4}
\|\bsg_{k+1}^0\|^2_{\bsQ+\kappa_1\bsK}
\le\Big(\frac{1}{\rho_2(L)}+\kappa_1\Big)\|\bsg_{k+1}^0\|^2.
\end{align}

Also note that
\begin{align}\label{sguT:v2k5}
\Big\|\bm{v}_k+\frac{1}{\beta_k}\bsg_{k}^0\Big\|^2_{\bsQ+\kappa_1\bsK}
\le\Big(\frac{1}{\rho_2(L)}+\kappa_1\Big)\Big\|\bm{v}_k+\frac{1}{\beta_k}\bsg_{k}^0
\Big\|^2_{\bsK}.
\end{align}

Then, from \eqref{sguT:v2k1}--\eqref{sguT:v2k5}, we have \eqref{sguT:v2k}.
\end{proof}

\begin{lemma}
Suppose Assumptions~\ref{sgd:assgraph} and \ref{sgd:assfiu}--\ref{sgd:ass:stochastic-grad:variance} hold, and $\{\beta_k\}$ in non-decreasing. Then the following holds for Algorithm~\ref{sguT:algorithm-sg}
\begin{align}
\mathbf{E}_{\mathfrak{F}_k}[W_{3,k+1}]
&\le W_{3,k}
-(1+\omega_k)\eta_k\alpha_k\bm{x}_k^\top\bsL\Big(\bm{v}_k+\frac{1}{\beta_k}\bsg_{k}^0\Big)
+\|\bm{x}_k\|^2_{\eta_k(\beta_k\bsL+\frac{1}{2}\bsK)
+\eta^2_k(\frac{1}{2}\alpha^2_k-\alpha_k\beta_k+\beta^2_k)\bsL^2}\nonumber\\
&\quad+\|\bsx_k\|^2_{\frac{1}{2}\omega_k\eta_k\alpha_k\bsL^2+\frac{1}{2}
\eta_k(1+3\eta_k)L_f^2\bsK}
+\frac{\eta_k}{2\beta^2_k}(1+3\eta_k)L_f^2\mathbf{E}_{\mathfrak{F}_k}[\|\bar{\bsg}^u_{k}\|^2]
+n\sigma^2\eta^2_k\nonumber\\
&\quad-\Big\|\bm{v}_k+\frac{1}{\beta_k}\bsg_{k}^0\Big\|^2_{\eta_k(\beta_k-\frac{1}{2}
-\eta_k\beta^2_k-\frac{1}{2}\omega_k\alpha_k)\bsK}
+\frac{1}{2}\omega_k\mathbf{E}_{\mathfrak{F}_k}[2W_{1,k+1}+\|\bsg_{k+1}^0\|^2],\label{sguT:v3k}
\end{align}
where $W_{3,k}=\bsx_k^\top\bsK(\bm{v}_k+\frac{1}{\beta_k}\bsg_k^0)$.
\end{lemma}
\begin{proof}
We have
\begin{align}
W_{3,k+1}&=\bsx_{k+1}^\top\bsK\Big(\bm{v}_{k+1}+\frac{1}{\beta_{k+1}}\bsg_{k+1}^0\Big)\nonumber\\
&=\bsx_{k+1}^\top\bsK\Big(\bm{v}_{k+1}+\frac{1}{\beta_{k}}\bsg_{k+1}^0
+\Big(\frac{1}{\beta_{k+1}}-\frac{1}{\beta_{k}}\Big)\bsg_{k+1}^0\Big)\nonumber\\
&=\bsx_{k+1}^\top\bsK\Big(\bm{v}_{k+1}+\frac{1}{\beta_{k}}\bsg_{k+1}^0\Big)
-\omega_k\bsx_{k+1}^\top\bsK\bsg_{k+1}^0\nonumber\\
&\le\bsx_{k+1}^\top\bsK\Big(\bm{v}_{k+1}+\frac{1}{\beta_{k}}\bsg_{k+1}^0\Big)
+\frac{1}{2}\omega_k(\|\bsx_{k+1}\|^2_{\bsK}+\|\bsg_{k+1}^0\|^2).\label{sguT:v3k1}
\end{align}
For the first term in the right-hand side of \eqref{sguT:v3k1}, we have
\begin{align}
\mathbf{E}_{\mathfrak{F}_k}\Big[\bsx_{k+1}^\top\bsK\Big(\bm{v}_{k+1}
+\frac{1}{\beta_k}\bsg_{k+1}^0\Big)\Big]
&=\mathbf{E}_{\mathfrak{F}_k}\Big[(\bm{x}_k-\eta_k(\alpha_k\bsL\bm{x}_k+\beta_k\bm{v}_k
+\bsg_k^0+\bsg^u_k-\bsg_k^0))^\top
\nonumber\\
&\quad\times\bsK\Big(\bm{v}_k+\frac{1}{\beta_k}\bsg_{k}^0+\eta_k\beta_k\bsL\bm{x}_k
+\frac{1}{\beta_k}\Big(\bsg_{k+1}^0-\bsg_{k}^0\Big)\Big)\Big]\nonumber\\
&=\bm{x}_k^\top(\bsK-\eta_k(\alpha_k+\eta_k\beta^2_k)\bsL)\Big(\bm{v}_k
+\frac{1}{\beta_k}\bsg_{k}^0\Big)\nonumber\\
&\quad+\|\bm{x}_k\|^2_{\eta_k\beta_k(\bsL-\eta_k\alpha_k\bsL^2)}\nonumber\\
&\quad+\frac{1}{\beta_k}\bm{x}_k^\top(\bsK-\eta_k\alpha_k\bsL)
\mathbf{E}_{\mathfrak{F}_k}[\bsg_{k+1}^0-\bsg_{k}^0]\nonumber\\
&\quad-\eta_k\beta_k\Big\|\bm{v}_k+\frac{1}{\beta_k}\bsg_{k}^0\Big\|^2_{\bsK}\nonumber\\
&\quad-\eta_k\Big(\bm{v}_k+\frac{1}{\beta_k}\bsg_{k}^0\Big)^\top\bsK
\mathbf{E}_{\mathfrak{F}_k}[\bsg_{k+1}^0-\bsg_{k}^0]\nonumber\\
&\quad-\eta_k(\bsg_k-\bsg_k^0)^\top
\bsK\Big(\bm{v}_k+\frac{1}{\beta_k}\bsg_{k}^0+\eta_k\beta_k\bsL\bm{x}_k\Big)\nonumber\\
&\quad-\frac{1}{\beta_k}\mathbf{E}_{\mathfrak{F}_k}[\eta_k(\bsg^u_k-\bsg_k^0)^\top
\bsK(\bsg_{k+1}^0-\bsg_{k}^0)]\nonumber\\
&\le\bm{x}_k^\top(\bsK-\eta_k\alpha_k\bsL)\Big(\bm{v}_k+\frac{1}{\beta_k}\bsg_{k}^0\Big)
+\frac{1}{2}\eta^2_k\beta^2_k\|\bsL\bsx_k\|^2\nonumber\\
&\quad+\frac{1}{2}\eta^2_k\beta^2_k\Big\|\bm{v}_k+\frac{1}{\beta_k}\bsg_{k}^0\Big\|^2_{\bsK}
+\|\bm{x}_k\|^2_{\eta_k\beta_k(\bsL-\eta_k\alpha_k\bsL^2)}\nonumber\\
&\quad+\frac{1}{2}\eta_k\|\bm{x}_k\|^2_\bsK
+\frac{1}{2\eta_k\beta^2_k}\mathbf{E}_{\mathfrak{F}_k}[\|\bsg_{k+1}^0-\bsg_{k}^0\|^2\nonumber\\
&\quad+\frac{1}{2}\eta^2_k\alpha^2_k\|\bsL\bm{x}_k\|^2
+\frac{1}{2\beta^2_k}\mathbf{E}_{\mathfrak{F}_k}[\|\bsg_{k+1}^0-\bsg_{k}^0\|^2]\nonumber\\
&\quad-\eta_k\beta_k\Big\|\bm{v}_k+\frac{1}{\beta_k}\bsg_{k}^0\Big\|^2_{\bsK}\nonumber\\
&\quad+\frac{1}{2}\eta^2_k\beta_k^2\Big\|\bm{v}_k+\frac{1}{\beta_k}\bsg_{k}^0\Big\|^2_{\bsK}
+\frac{1}{2\beta_k^2}\mathbf{E}_{\mathfrak{F}_k}[\|\bsg_{k+1}^0-\bsg_{k}^0\|^2]\nonumber\\
&\quad+\frac{1}{2}\eta_k\|\bsg_k-\bsg_k^0\|^2
+\frac{1}{2}\eta_k\Big\|\bm{v}_k+\frac{1}{\beta_k}\bsg_{k}^0\Big\|^2_{\bsK}\nonumber\\
&\quad+\frac{1}{2}\eta^2_k\|\bsg_k-\bsg_k^0\|^2
+\frac{1}{2}\eta^2_k\beta^2_k\|\bsL\bm{x}_k\|^2\nonumber\\
&\quad+\frac{1}{2}\eta^2_k\mathbf{E}_{\mathfrak{F}_k}[\|\bsg^u_k-\bsg_k^0\|^2]
+\frac{1}{2\beta^2_k}\mathbf{E}_{\mathfrak{F}_k}[\|\bsg_{k+1}^0-\bsg_{k}^0\|^2]\nonumber\\
&=\bm{x}_k^\top(\bsK-\eta_k\alpha_k\bsL)\Big(\bm{v}_k+\frac{1}{\beta_k}\bsg_{k}^0\Big)\nonumber\\
&\quad+\frac{1}{2}(\eta_k+\eta^2_k)\|\bsg_k-\bsg_k^0\|^2
+\frac{1}{2}\eta^2_k\mathbf{E}_{\mathfrak{F}_k}[\|\bsg^u_k-\bsg_k^0\|^2]\nonumber\\
&\quad+\|\bm{x}_k\|^2_{\eta_k(\beta_k\bsL+\frac{1}{2}\bsK)
+\eta^2_k(\frac{1}{2}\alpha^2_k-\alpha_k\beta_k+\beta^2_k)\bsL^2}\nonumber\\
&\quad+\Big(\frac{1}{2\eta_k\beta^2_k}+\frac{3}{2\beta^2_k}\Big)
\mathbf{E}_{\mathfrak{F}_k}[\|\bsg_{k+1}^0-\bsg_{k}^0\|^2]\nonumber\\
&\quad-\Big\|\bm{v}_k+\frac{1}{\beta_k}\bsg_{k}^0\Big\|^2_{\eta_k(\beta_k-\frac{1}{2}
-\eta_k\beta^2_k)\bsK}\nonumber\\
&\le\bm{x}_k^\top\bsK\Big(\bm{v}_k+\frac{1}{\beta_k}\bsg_{k}^0\Big)
-(1+\omega_k)\eta_k\alpha_k\bm{x}_k^\top\bsL\Big(\bm{v}_k+\frac{1}{\beta_k}\bsg_{k}^0\Big)\nonumber\\
&\quad+\omega_k\eta_k\alpha_k\bm{x}_k^\top\bsL\Big(\bm{v}_k+\frac{1}{\beta_k}\bsg_{k}^0\Big)\nonumber\\
&\quad+\|\bm{x}_k\|^2_{\eta_k(\beta_k\bsL+\frac{1}{2}\bsK)
+\eta^2_k(\frac{1}{2}\alpha^2_k-\alpha_k\beta_k+\beta^2_k)\bsL^2
+\frac{1}{2}\eta_k(1+3\eta_k)L_f^2\bsK}\nonumber\\
&\quad+\frac{\eta_k}{2\beta^2_k}(1+3\eta_k)L_f^2\mathbf{E}_{\mathfrak{F}_k}[\|\bar{\bsg}^u_{k}\|^2]
+n\sigma^2\eta^2_k\nonumber\\
&\quad-\Big\|\bm{v}_k+\frac{1}{\beta_k}\bsg_{k}^0\Big\|^2_{\eta_k(\beta_k-\frac{1}{2}
-\eta_k\beta^2_k)\bsK}
,\label{sguT:v3k2}
\end{align}
where the first equality holds due to \eqref{sgu:kia-algo-dc-compact}; the second equality holds since \eqref{nonconvex:KL-L-eq}  in Lemma~\ref{nonconvex:lemma-Xinlei}, $\bsx_{k}$ and $\bsv_{k}$ are independent of $\mathfrak{F}_k$, and \eqref{sguT:sgproof-hg1}; the first inequality holds due to the Cauchy--Schwarz inequality, \eqref{nonconvex:KL-L-eq}, $\rho(\bsK)=1$, and the Jensen's inequality; and the last  inequality holds due to \eqref{sguT:gg1}, \eqref{sguT:sgproof-hg5}, and \eqref{sguT:gg}.
For the third term in the right-hand side of \eqref{sguT:v3k2}, we have
\begin{align}\label{sguT:v3k3}
\omega_k\eta_k\alpha_k\bm{x}_k^\top\bsL\Big(\bm{v}_k+\frac{1}{\beta_k}\bsg_{k}^0\Big)
&=\omega_k\eta_k\alpha_k\bsx^\top_k\bsL\bsK\Big(\bm{v}_k+\frac{1}{\beta_k}\bsg_k^0\Big)\nonumber\\
&\le\|\bsx_k\|^2_{\frac{1}{2}\omega_k\eta_k\alpha_k\bsL^2}
+\Big\|\bm{v}_k+\frac{1}{\beta_k}\bsg_k^0\Big\|^2_{\frac{1}{2}\omega_k\eta_k\alpha_k\bsK}.
\end{align}

Then, from \eqref{sguT:v3k1}--\eqref{sguT:v3k3}, we have \eqref{sguT:v3k}.
\end{proof}

\begin{lemma}
Suppose Assumptions~\ref{sgd:assoptset}--\ref{sgd:ass:stochastic-grad:mean} hold. Then the following holds for Algorithm~\ref{sguT:algorithm-sg}
\begin{align}
\mathbf{E}_{\mathfrak{F}_k}[W_{4,k+1}]
&\le  W_{4,k}-\frac{\eta_k}{4}\|\bar{\bsg}_{k}\|^2
+\|\bsx_k\|^2_{\frac{\eta_k}{2}L_f^2\bsK}
-\frac{\eta_k}{4}\|\bar{\bsg}_{k}^0\|^2
+\frac{1}{2}\eta^2_kL_f\mathbf{E}_{\mathfrak{F}_k}[\|\bar{\bsg}^u_{k}\|^2],\label{sguT:v4k}
\end{align}
where $W_{4,k}=n(f(\bar{x}_k)-f^*)=\tilde{f}(\bar{\bsx}_k)-\tilde{f}^*$.
\end{lemma}
\begin{proof}
We first note that $W_{4,k}$ is well defined due to $f^*>-\infty$ as assumed in Assumption~\ref{sgd:assoptset}.

From \eqref{sguT:gg1} and $\rho(\bsH)=1$, we have that
\begin{align}
&\|\bar{\bsg}^0_k-\bar{\bsg}_k\|^2=\|\bsH(\bsg^0_{k}-\bsg_{k})\|^2
\le\|\bsg^0_{k}-\bsg_{k}\|^2\le L_f^2\|\bsx_k\|^2_{\bsK}.\label{sguT:gg2}
\end{align}

From \eqref{sguT:sgproof-hg1}, we have
\begin{align}
&\mathbf{E}_{\mathfrak{F}_k}[\bar{\bsg}^u_k]=\mathbf{E}_{\mathfrak{F}_k}[\bsH\bsg^u_k]
=\bsH\mathbf{E}_{\mathfrak{F}_k}[\bsg^u_k]=\bar{\bsg}_k.\label{sguT:sgproof-hg3}
\end{align}

We have
\begin{align}
\mathbf{E}_{\mathfrak{F}_k}[W_{4,k+1}]
&=\mathbf{E}_{\mathfrak{F}_k}[\tilde{f}(\bar{\bsx}_{k+1})-nf^*]\nonumber\\
&=\mathbf{E}_{\mathfrak{F}_k}[\tilde{f}(\bar{\bsx}_k)-nf^*+
\tilde{f}(\bar{\bsx}_{k+1})-\tilde{f}(\bar{\bsx}_k)]\nonumber\\
&\le\mathbf{E}_{\mathfrak{F}_k}[\tilde{f}(\bar{\bsx}_k)-nf^*
-\eta_k(\bar{\bsg}_{k}^u)^\top\bsg^0_k
+\frac{1}{2}\eta^2_kL_f\|\bar{\bsg}^u_{k}\|^2]\nonumber\\
&=\tilde{f}(\bar{\bsx}_k)-nf^*
-\eta_k\bar{\bsg}_{k}^\top\bsg^0_k
+\frac{1}{2}\eta^2_kL_f\mathbf{E}_{\mathfrak{F}_k}[\|\bar{\bsg}^u_{k}\|^2]\nonumber\\
&=\tilde{f}(\bar{\bsx}_k)-nf^*
-\eta_k\bar{\bsg}_{k}^\top\bar{\bsg}^0_k
+\frac{1}{2}\eta^2_kL_f\mathbf{E}_{\mathfrak{F}_k}[\|\bar{\bsg}^u_{k}\|^2]\nonumber\\
&=W_{4,k}
-\frac{\eta_k}{2}\bar{\bsg}_{k}^\top(\bar{\bsg}_k+\bar{\bsg}^0_k-\bar{\bsg}_k)
-\frac{\eta_k}{2}(\bar{\bsg}_{k}-\bar{\bsg}^0_k+\bar{\bsg}^0_k)^\top\bar{\bsg}^0_k
+\frac{1}{2}\eta^2_kL_f\mathbf{E}_{\mathfrak{F}_k}[\|\bar{\bsg}^u_{k}\|^2]\nonumber\\
&\le W_{4,k}-\frac{\eta_k}{4}\|\bar{\bsg}_{k}\|^2
+\frac{\eta_k}{4}\|\bar{\bsg}^0_k-\bar{\bsg}_k\|^2
-\frac{\eta_k}{4}\|\bar{\bsg}_{k}^0\|^2
+\frac{\eta_k}{4}\|\bar{\bsg}^0_k-\bar{\bsg}_k\|^2
+\frac{1}{2}\eta^2_kL_f\mathbf{E}_{\mathfrak{F}_k}[\|\bar{\bsg}^u_{k}\|^2]\nonumber\\
&=W_{4,k}-\frac{\eta_k}{4}\|\bar{\bsg}_{k}\|^2
+\frac{\eta_k}{2}\|\bar{\bsg}^0_k-\bar{\bsg}_k\|^2
-\frac{\eta_k}{4}\|\bar{\bsg}_{k}^0\|^2
+\frac{1}{2}\eta^2_kL_f\mathbf{E}_{\mathfrak{F}_k}[\|\bar{\bsg}^u_{k}\|^2],
\label{sguT:v4k-1}
\end{align}
where the first inequality holds since that $\tilde{f}$ is smooth, \eqref{nonconvex:lemma:lipschitz} and \eqref{sguT:xbardynamic-sg}; the third equality holds since $\bsx_{k}$ and $\bsv_{k}$ are independent of $\mathfrak{F}_k$ and \eqref{sguT:sgproof-hg3}; the fourth equality holds due to $\bar{\bsg}_{k}^\top\bsg^0_k=\bsg_{k}^\top\bsH\bsg^0_k=\bsg_{k}^\top\bsH\bsH\bsg^0_k
=\bar{\bsg}_{k}^\top\bar{\bsg}^0_k$; and the second inequality holds due to the Cauchy--Schwarz inequality.

Then, from \eqref{sguT:gg2} and \eqref{sguT:v4k-1}, we have \eqref{sguT:v4k}.
\end{proof}

\subsection{Proof of Theorem~\ref{sguT:thm-sg-smT}}\label{sguT:proof-thm-sg-smT}
We denote the following notations.
\begin{align*}
c_0(\kappa_1,\kappa_2)&=\max\{4\kappa_2\varepsilon_5,~\varepsilon_{6}\},\\
c_1&=\frac{1}{\rho_2(L)}+1,\\
c_2(\kappa_1)&=\min\Big\{\frac{\varepsilon_1}{\varepsilon_2},~\frac{1}{5}\Big\},\\
\kappa_3&=\frac{1}{\rho_2(L)}+\kappa_1+1,\\
\kappa_4&=\frac{1}{\rho_2(L)}+\kappa_1+\frac{3}{2},\\
\kappa_5&=\frac{\kappa_1+1}{2}+\frac{1}{2\rho_2(L)},\\
\kappa_6&=\min\Big\{\frac{1}{2\rho(L)},~\frac{\kappa_1-1}{2\kappa_1}\Big\},\\
\varepsilon_1&=(\kappa_1-1)\rho_2(L)-1,\\
\varepsilon_2&=\rho(L)+(2\kappa_1^2+1)\rho(L^2)+1,\\
\varepsilon_3&=\varepsilon_1\kappa_2-\varepsilon_2\kappa_2^2,\\
\varepsilon_4&=\frac{1}{2}(\kappa_2-5\kappa_2^2),\\
\varepsilon_5&=L_f+\frac{1}{\kappa_2\varepsilon_{6}}\kappa_3L_f^2
+\frac{2}{\varepsilon_{6}^2}\kappa_4L_f^2,\\
\varepsilon_6&=\max\Big\{\frac{1}{2}(2+3L_f^2),~\kappa_3\Big\}.
\end{align*}

To prove Theorem~\ref{sguT:thm-sg-smT}, we  need the following lemma.


\begin{lemma}\label{sguT:lemma:sg2}
Suppose Assumptions~\ref{sgd:assgraph}--\ref{sgd:ass:stochastic-grad:variance} hold. Suppose $\alpha_k=\alpha=\kappa_1\beta$, $\beta_k=\beta\ge c_0(\kappa_1,\kappa_2)$, and $\eta_k=\eta=\kappa_2/\beta$, where $\kappa_1>c_1$ and $\kappa_2\in(0,c_2(\kappa_1))$. Then, for any $k\in\mathbb{N}_0$ the following holds for Algorithm~\ref{sguT:algorithm-sg}
\begin{subequations}
\begin{align}
\mathbf{E}_{\mathfrak{F}_k}[W_{k+1}]
& \le W_{k}-\|\bsx_k\|^2_{\varepsilon_3\bsK}
-\Big\|\bm{v}_k+\frac{1}{\beta}\bsg_{k}^0\Big\|^2_{\varepsilon_4\bsK}
-\frac{1}{4}\eta\|\bar{\bsg}^0_{k}\|^2
+(\varepsilon_5+3n)\sigma^2\eta^2,\label{sguT:vkLya2}\\
\mathbf{E}_{\mathfrak{F}_k}[\breve{W}_{k+1}]
& \le \breve{W}_{k}-\|\bsx_k\|^2_{\varepsilon_3\bsK}
-\Big\|\bm{v}_k+\frac{1}{\beta}\bsg_{k}^0\Big\|^2_{\varepsilon_4\bsK}
+2\varepsilon_5\eta^2\|\bar{\bsg}_{k}^0\|^2\nonumber\\
&\quad+2L_f^2\varepsilon_5\eta^2\|\bsx_k\|^2_{\bsK}
+(\varepsilon_5+3n)\sigma^2\eta^2,\label{sguT:vkLya2-1-3-speed}\\
\mathbf{E}_{\mathfrak{F}_k}[W_{4,k+1}]
&\le W_{4,k}-\frac{1}{4}\eta\|\bar{\bsg}_{k}^0\|^2+\|\bsx_k\|^2_{\frac{1}{2}\eta L_f^2\bsK}+L_f\sigma^2\eta^2,\label{sguT:v4k-speed}
\end{align}
\end{subequations}
where $W_{k}=\sum_{i=1}^{4}W_{i,k}$ and $\breve{W}_k=\sum_{i=1}^{3}W_{i,k}$.
\end{lemma}
\begin{proof}
(i) Noting that $\alpha_k=\alpha=\kappa_1\beta$, $\beta_k=\beta$, $\eta_k=\eta$, and $\omega_k=\frac{1}{\beta_{k}}-\frac{1}{\beta_{k+1}}=0$, from \eqref{sguT:v1k}, \eqref{sguT:v2k}, \eqref{sguT:v3k}, and \eqref{sguT:v4k},  we have
\begin{align}\label{sguT:vkLya-1}
\mathbf{E}_{\mathfrak{F}_k}[W_{k+1}]
&\le W_{k}
+\Big\|\bm{v}_k+\frac{1}{\beta}\bsg_k^0\Big\|^2_{\frac{3}{2}\eta^2\beta^2\bsK}
+2n\sigma^2\eta^2\nonumber\\
&\quad-\|\bsx_k\|^2_{\eta\alpha\bsL-\frac{1}{2}\eta\bsK
-\frac{3}{2}\eta^2\alpha^2\bsL^2-\frac{1}{2}\eta(1+5\eta)L_f^2\bsK}\nonumber\\
&\quad+\|\bsx_k\|^2_{\eta^2\beta^2(\bsL+\kappa_1\bsL^2)}
+\frac{1}{2}\eta\Big(\frac{1}{\rho_2(L)}+\kappa_1\Big)
\Big\|\bm{v}_k+\frac{1}{\beta}\bsg_{k}^0\Big\|^2_{\bsK}\nonumber\\
&\quad+\frac{\eta}{\beta^2}\Big(\eta+\frac{1}{2}\Big)
\Big(\frac{1}{\rho_2(L)}+\kappa_1\Big)L_f^2\mathbf{E}_{\mathfrak{F}_k}[\|\bar{\bsg}^u_k\|^2]\nonumber\\
&\quad+\|\bm{x}_k\|^2_{\eta(\beta\bsL+\frac{1}{2}\bsK)
+\eta^2(\frac{1}{2}\alpha^2-\alpha\beta+\beta^2)\bsL^2
+\frac{1}{2}\eta(1+3\eta)L_f^2\bsK}\nonumber\\
&\quad+\frac{\eta}{2\beta^2}(1+3\eta)L_f^2\mathbf{E}_{\mathfrak{F}_k}[\|\bar{\bsg}^u_{k}\|^2]
+n\sigma^2\eta^2\nonumber\\
&\quad-\Big\|\bm{v}_k+\frac{1}{\beta}\bsg_{k}^0\Big\|^2_{\eta(\beta-\frac{1}{2}-\eta\beta^2)\bsK}
-\frac{1}{4}\eta\|\bar{\bsg}_{k}\|^2\nonumber\\
&\quad+\|\bsx_k\|^2_{\frac{1}{2}\eta L_f^2\bsK}-\frac{1}{4}\eta\|\bar{\bsg}_{k}^0\|^2
+\frac{1}{2}\eta^2L_f\mathbf{E}_{\mathfrak{F}_k}[\|\bar{\bsg}^u_{k}\|^2].
\end{align}

Note that
\begin{align}
\mathbf{E}_{\mathfrak{F}_k}[\|\bar{\bsg}^u_k\|^2]
&=\mathbf{E}_{\mathfrak{F}_k}[\|\bar{\bsg}^u_k-\bar{\bsg}_k+\bar{\bsg}_k\|^2]\nonumber\\
&\le2\mathbf{E}_{\mathfrak{F}_k}[\|\bar{\bsg}^u_k-\bar{\bsg}_k\|^2]+2\|\bar{\bsg}_k\|^2\nonumber\\
&=2n\mathbf{E}_{\mathfrak{F}_k}[\|\frac{1}{n}\sum_{i=1}^{n}(g^u_{i,k}-g_{i,k})\|^2]
+2\|\bar{\bsg}_k\|^2\nonumber\\
&=\frac{2}{n}\mathbf{E}_{\mathfrak{F}_k}[\|\sum_{i=1}^{n}(g^u_{i,k}-g_{i,k})\|^2]
+2\|\bar{\bsg}_k\|^2\nonumber\\
&=\frac{2}{n}\sum_{i=1}^{n}\mathbf{E}_{\mathfrak{F}_k}[\|g^u_{i,k}-g_{i,k}\|^2]
+2\|\bar{\bsg}_k\|^2\nonumber\\
&\le2\sigma^2+2\|\bar{\bsg}_k\|^2,\label{sguT:vkLya-3}
\end{align}
where the first inequality holds due to the Cauchy--Schwarz inequality; the last equality holds since $\{g^u_{i,k},~i\in[n]\}$ are independent of each other as assumed in Assumption~\ref{sgd:ass:stochastic-grad:xi}, $\bsx_{k}$ and $\bsv_{k}$ are independent of $\mathfrak{F}_k$, and $\mathbf{E}_{\mathfrak{F}_k}[g^u_{i,k}]=g_{i,k}$ as assumed in Assumption~\ref{sgd:ass:stochastic-grad:mean}; and the last inequality holds due to \eqref{sguT:sgproof-hg2}.

From \eqref{sguT:vkLya-1}, \eqref{sguT:vkLya-3}, and $\alpha=\kappa_1\beta$, we have
\begin{align}\label{sguT:vkLya-4}
\mathbf{E}_{\mathfrak{F}_k}[W_{k+1}]
&\le W_{k}-\|\bsx_k\|^2_{\eta\bsM_{1}-\eta^2\bsM_{2}}
-\Big\|\bm{v}_k+\frac{1}{\beta}\bsg_{k}^0\Big\|^2_{b_{1,k}\bsK}\nonumber\\
&\quad-b_{2,k}\eta \|\bar{\bsg}_{k}\|^2
-\frac{1}{4}\eta\|\bar{\bsg}^0_{k}\|^2+b_{3,k}\sigma^2\eta^2+3n\sigma^2\eta^2,
\end{align}
where
\begin{align*}
\bsM_{1}&=(\alpha-\beta)\bsL-\frac{1}{2}(2+3L_f^2)\bsK,\\
\bsM_{2}&=\beta^2\bsL+(2\alpha^2+\beta^2)\bsL^2+4L_f^2\bsK,\\
b_{1,k}&=\frac{1}{2}(2\beta-\kappa_3)\eta
-\frac{5}{2}\beta^2\eta^2,\\
b_{2,k}&=\frac{1}{4}-b_{3,k}\eta,\\
b_{3,k}&=L_f+\frac{1}{\beta^2\eta}\kappa_3L_f^2+\frac{2}{\beta^2}\kappa_4L_f^2.
\end{align*}

From \eqref{nonconvex:KL-L-eq2}, $\alpha=\kappa_1\beta$, $\kappa_1>c_1>1$, $\eta=\kappa_2/\beta$, and $\beta\ge c_0(\kappa_1,\kappa_2)\ge\varepsilon_{6}\ge(2+3L_f^2)/2$,  we have
\begin{align}\label{sguT:vkLya-M1}
\eta\bsM_{1}\ge\varepsilon_1\kappa_2\bsK.
\end{align}

From \eqref{nonconvex:KL-L-eq2}, $\alpha=\kappa_1\beta$, and $\beta\ge\frac{1}{2}(2+3L_f^2)>2L_f$,  we have
\begin{align}\label{sguT:vkLya-M2}
\eta^2\bsM_{2}\le\varepsilon_2\kappa_2^2\bsK.
\end{align}

From $\beta\ge\kappa_3$, we have
\begin{align}\label{sguT:vkLya-b1}
b_{1,k}\ge&\varepsilon_4.
\end{align}

From $\kappa_1>c_1=1/\rho_2(L)+1$, we have
\begin{align}
\varepsilon_1>0.\label{sguT:varepsilon1and4}
\end{align}

From \eqref{sguT:varepsilon1and4} and $\kappa_2\in(0,\min\{\frac{\varepsilon_1}{\varepsilon_2},~\frac{1}{5}\})$, we have
\begin{subequations}
\begin{align}
&\varepsilon_3>0,\label{sguT:kappa2-1}\\
&\varepsilon_4>0.\label{sguT:kappa2-2}
\end{align}
\end{subequations}

From \eqref{sguT:kappa2-1}, \eqref{sguT:kappa2-2}, and $\beta\ge4\kappa_2\varepsilon_5$, we have
\begin{subequations}
\begin{align}
&b_{3,k}=L_f+\frac{1}{\beta^2\eta_k}\kappa_3L_f^2+\frac{2}{\beta^2}\kappa_4L_f^2
\le\varepsilon_5,\label{sguT:varepsilon9}\\
&b_{2,k}=\frac{1}{4}-b_{3,k}\eta\ge\frac{1}{4}-\frac{\kappa_2}{\beta}\varepsilon_5\ge0.
\label{sguT:varepsilon13}
\end{align}
\end{subequations}

From \eqref{sguT:vkLya-4}--\eqref{sguT:vkLya-b1}, \eqref{sguT:varepsilon9}, and \eqref{sguT:varepsilon13}, we have \eqref{sguT:vkLya2}.

(ii) Similar to the way to get \eqref{sguT:vkLya2}, we have
\begin{align}
\mathbf{E}_{\mathfrak{F}_k}[\breve{W}_{k+1}]
 \le \breve{W}_{k}-\|\bsx_k\|^2_{\varepsilon_3\bsK}
-\Big\|\bm{v}_k+\frac{1}{\beta}\bsg_{k}^0\Big\|^2_{\varepsilon_4\bsK}
+\varepsilon_5\eta^2\|\bar{\bsg}_{k}\|^2
+(\varepsilon_5+3n)\sigma^2\eta^2,\label{sguT:vkLya2-1-3-speed-1}
\end{align}

We have
\begin{align}
\|\bar{\bsg}_{k}\|^2
=\|\bar{\bsg}_{k}-\bar{\bsg}_{k}^0+\bar{\bsg}_{k}^0\|^2\le2\|\bar{\bsg}_{k}-\bar{\bsg}_{k}^0\|^2+2\|\bar{\bsg}_{k}^0\|^2
\le2L_f^2\|\bsx_k\|^2_{\bsK}+2\|\bar{\bsg}_{k}^0\|^2,\label{sgu:vkLya-2g0kbar}
\end{align}
where the last inequality holds due to \eqref{sguT:gg2}.

From \eqref{sguT:vkLya2-1-3-speed-1} and \eqref{sgu:vkLya-2g0kbar}, we have \eqref{sguT:vkLya2-1-3-speed}.

(iii) From \eqref{sguT:v4k} and \eqref{sguT:vkLya-3}, we have
\begin{align}
\mathbf{E}_{\mathfrak{F}_k}[W_{4,k+1}]\le W_{4,k}-\frac{1}{4}\eta\|\bar{\bsg}_{k}\|^2
+\|\bsx_k\|^2_{\frac{1}{2}\eta L_f^2\bsK}-\frac{1}{4}\eta\|\bar{\bsg}_{k}^0\|^2
+\eta^2L_f(\sigma^2+\|\bar{\bsg}_k\|^2),\label{sguT:v4k-speed0}
\end{align}

From $\eta=\kappa_2/\beta$ and $\beta\ge4\kappa_2\varepsilon_5>4\kappa_2L_f$, we have
\begin{align}\label{sguT:v4k-speed1}
\eta L_f<\frac{1}{4}.
\end{align}

From \eqref{sguT:v4k-speed0} and \eqref{sguT:v4k-speed1}, we have \eqref{sguT:v4k-speed}.
\end{proof}

Now we are ready to prove Theorem~\ref{sguT:thm-sg-smT}.

Denote
\begin{align*}
\hat{V}_k=\|\bm{x}_k\|^2_{\bsK}+\Big\|\bsv_k
+\frac{1}{\beta_k}\bsg_k^0\Big\|^2_{\bsK}+n(f(\bar{x}_k)-f^*).
\end{align*}
We have
\begin{align}
W_{k}
&=\frac{1}{2}\|\bsx_{k}\|^2_{\bsK}
+\frac{1}{2}\Big\|\bsv_k+\frac{1}{\beta}\bsg_k^0\Big\|^2_{\bsQ+\kappa_1\bsK}
+\bsx_k^\top\bsK\Big(\bm{v}_k+\frac{1}{\beta}\bsg_k^0\Big)+n(f(\bar{x}_k)-f^*)\nonumber\\
&\ge\frac{1}{2}\|\bsx_{k}\|^2_{\bsK}
+\frac{1}{2}\Big(\frac{1}{\rho(L)}+\kappa_1\Big)\Big\|\bsv_k+\frac{1}{\beta}
\bsg_k^0\Big\|^2_{\bsK}-\frac{1}{2\kappa_1}\|\bsx_{k}\|^2_{\bsK}
-\frac{\kappa_1}{2}\Big\|\bsv_k+\frac{1}{\beta}\bsg_k^0\Big\|^2_{\bsK}
+n(f(\bar{x}_k)-f^*)\nonumber\\
&\ge\kappa_6\Big(\|\bsx_{k}\|^2_{\bsK}+\Big\|\bsv_k+\frac{1}{\beta}\bsg_k^0\Big\|^2_{\bsK}\Big)
+n(f(\bar{x}_k)-f^*)\label{sguT:vkLya3.2}\\
&\ge\kappa_6\hat{V}_k\ge0,\label{sguT:vkLya3}
\end{align}
where the first inequality holds due to \eqref{nonconvex:lemma-eq2} and the Cauchy--Schwarz inequality; and the last inequality holds due to $0<\kappa_6<0.5$. Similarly, we have
\begin{align}\label{sguT:vkLya3.1}
W_k\le\kappa_5\hat{V}_k.
\end{align}

From \eqref{sguT:vkLya2} and \eqref{sguT:kappa2-2}, we have
\begin{align}\label{sguT:vkLya4}
\mathbf{E}_{\mathfrak{F}_k}[W_{k+1}]\le W_{k}-\varepsilon_3\|\bsx_k\|^2_{\bsK}
-\frac{\kappa_2}{4\beta}\|\bar{\bsg}^0_{k}\|^2
+\frac{(\varepsilon_5+3n)\kappa_2^2\sigma^2}{\beta^2}.
\end{align}
Then, taking expectation in $\calF_{T}$ and summing \eqref{sguT:vkLya4} over $ k\in[0,T]$  yield
\begin{align}\label{sguT:vkLya4.1}
\mathbf{E}[W_{T+1}]+\sum_{k=0}^{T}\mathbf{E}\Big[\varepsilon_3\|\bsx_k\|^2_{\bsK}
+\frac{\kappa_2}{4\beta}\|\bar{\bsg}^0_{k}\|^2\Big]\le W_{0}+\frac{(T+1)(\varepsilon_5+3n)\kappa_2^2\sigma^2}{\beta^2}.
\end{align}

From \eqref{sguT:vkLya4.1}, \eqref{sguT:vkLya3}, and \eqref{sguT:kappa2-1}, we have
\begin{align}\label{sguT:thm-sg-sm-equ1.1p}
\frac{1}{T+1}\sum_{k=0}^{T}\mathbf{E}\Big[\frac{1}{n}\sum_{i=1}^{n}\|x_{i,k}-\bar{x}_k\|^2\Big]
=\frac{1}{n(T+1)}\sum_{k=0}^{T}\mathbf{E}[\|\bsx_k\|^2_{\bsK}]
\le\frac{W_{0}}{n\varepsilon_3(T+1)}
+\frac{(\varepsilon_5+3n)\kappa_2^2\sigma^2}{n\varepsilon_3\beta^2}.
\end{align}
Noting that $W_0=\mathcal{O}(n)$, from \eqref{sguT:thm-sg-sm-equ1.1p}, we have \eqref{sguT:thm-sg-sm-equ3.1}.

Taking expectation in $\calF_{T}$ and summing \eqref{sguT:v4k-speed} over $ k\in[0,T]$  yield
\begin{align}\label{sguT:vkLya4.1-speed}
&\frac{1}{4}n\sum_{k=0}^{T}\mathbf{E}[\|\nabla f(\bar{x}_k)\|^2]=\frac{1}{4}\sum_{k=0}^{T}\mathbf{E}[\|\bar{\bsg}^0_{k}\|^2]
\le \frac{W_{4,0}}{\eta}+\frac{1}{2}L_f^2\sum_{k=0}^{T}\mathbf{E}[\|\bsx_k\|^2_{\bsK}]
+(T+1)L_f\sigma^2\eta.
\end{align}

From \eqref{sguT:vkLya4.1-speed}, $\eta=\kappa_2/\beta=\sqrt{n}/\sqrt{T}$, and \eqref{sguT:thm-sg-sm-equ1.1p}, we have
\begin{align*}
&\frac{1}{T}\sum_{k=0}^{T-1}\mathbf{E}[\|\nabla f(\bar{x}_k)\|^2]
\le\frac{4\beta}{\kappa_2T}(f(\bar{x}_0)-f^*)
+\frac{4L_f\sigma^2\kappa_2}{n\beta}
+\mathcal{O}(\frac{1}{T})+\mathcal{O}(\frac{1}{\beta^2}),
\end{align*}
which gives \eqref{sguT:thm-sg-sm-equ3}.

Taking expectation in $\calF_{T}$ and summing \eqref{sguT:v4k-speed} over $ k\in[0,T]$  yield
\begin{align}\label{sguT:thm-sg-sm-equ2p}
&n(\mathbf{E}[f(\bar{x}_{T+1})]-f^*)=\mathbf{E}[W_{4,T+1}]
\le W_{4,0}+\frac{1}{2}\eta L_f^2\sum_{k=0}^{T}\mathbf{E}[\|\bsx_k\|^2_{\bsK}]
+L_f\sigma^2\eta^2(T+1).
\end{align}

From \eqref{sguT:vkLya4.1}, \eqref{sguT:thm-sg-sm-equ2p}, and $\eta=\kappa_2/\beta$, we have \eqref{sguT:thm-sg-sm-equ4}.

\subsection{Proof of Theorem~\ref{sguT:thm-sg-diminishingT}}\label{sguT:proof-thm-sg-diminishingT}
In addition to the notations defined in Appendix~\ref{sguT:proof-thm-sg-smT}, we also denote
\begin{align*}
\varepsilon_7&=\frac{1}{\kappa_5}\min\Big\{\varepsilon_3,~\varepsilon_4,
~\frac{\nu}{2(T+1)^\theta}\Big\}.
\end{align*}

From the conditions in Theorem~\ref{sguT:thm-sg-diminishingT}, we know that all conditions needed in Lemma~\ref{sguT:lemma:sg2} are satisfied, so \eqref{sguT:vkLya2}--\eqref{sguT:v4k-speed} still hold.

From Assumptions~\ref{sgd:assoptset} and \ref{sgd:assfil} as well as \eqref{nonconvex:equ:plc}, we have that
\begin{align}\label{nonconvex:gg3}
\|\bar{\bsg}^0_k\|^2=n\|\nabla f(\bar{x}_k)\|^2\ge2\nu n(f(\bar{x}_k)-f^*)=2\nu W_{4,k}.
\end{align}

From \eqref{sguT:vkLya3}, we have
\begin{align}\label{sguT:vkLya5}
\|\bsx_{k}\|^2_{\bsK}+W_{4,k}
\le\hat{V}_k\le\frac{W_{k}}{\kappa_6}.
\end{align}

From \eqref{sguT:vkLya2}, \eqref{nonconvex:gg3}, \eqref{sguT:vkLya3}, \eqref{sguT:vkLya3.1}, and \eqref{sguT:step:eta1T}, we have
\begin{align}\label{sguT:vkLya2-plT-constant}
\mathbf{E}_{\mathfrak{F}_k}[W_{k+1}]
&\le W_{k}-\varepsilon_3\|\bsx_k\|^2_{\bsK}
-\varepsilon_4\Big\|\bm{v}_k+\frac{1}{\beta_k}\bsg_{k}^0\Big\|^2_{\bsK}
-\frac{1}{2}\eta\nu W_{4,k}
+(\varepsilon_5+3n)\sigma^2\eta^2\nonumber\\
&\le W_{k}-\frac{1}{\kappa_5}\min\Big\{\varepsilon_3,~\varepsilon_4,
~\frac{\nu\eta}{2}\Big\}W_{k}
+(\varepsilon_5+3n)\sigma^2\eta^2.
\end{align}

From \eqref{sguT:vkLya2-plT-constant} and \eqref{sguT:step:eta1T}, we have
\begin{align}\label{sguT:vkLya2-plT}
\mathbf{E}_{\mathfrak{F}_k}[W_{k+1}]
\le W_{k}-\varepsilon_7W_{k}
+\frac{(\varepsilon_5+3n)\sigma^2}{(T+1)^{2\theta}},~\forall k\le T.
\end{align}

From $\kappa_1>1$, we have $\kappa_5>1$. From $0<\kappa_2<1/5$, we have $\varepsilon_4=(\kappa_2-5\kappa_2^2)/2\le1/40$.
Thus,
\begin{align}\label{sguT:vkLya2-pl-r1.1T}
0<\varepsilon_7\le\frac{\varepsilon_4}{\kappa_5}\le\frac{1}{40}.
\end{align}

Then, from \eqref{sguT:vkLya2-plT}, \eqref{sguT:vkLya3}, and \eqref{sguT:vkLya2-pl-r1.1T}, we have
\begin{align}\label{sguT:vkLya2-plT1}
\mathbf{E}[W_{k+1}]
&\le (1-\varepsilon_7)^{k+1}W_{0}
+\frac{(\varepsilon_5+3n)\sigma^2}{(T+1)^{2\theta}}
\sum_{l=0}^{k}(1-\varepsilon_7)^l\nonumber\\
&\le (1-\varepsilon_7)^{k+1}W_{0}
+\frac{(\varepsilon_5+3n)\sigma^2}
{\varepsilon_7(T+1)^{2\theta}},~\forall k\le T.
\end{align}

Then, noting that $\varepsilon_7=\mathcal{O}(1/(T+1)^{\theta})$ and $\theta\in(0,1)$, from \eqref{sguT:vkLya2-plT1}, \eqref{sgu:lemma:exponential-equ}, and \eqref{sguT:vkLya5}, we have \begin{align}\label{sguT:thm-sg-diminishing-equ1T}
\mathbf{E}[\|\bsx_k\|^2_{\bsK}+W_{4,k}]
=\mathcal{O}(\frac{n}{T^\theta}),~\forall k\le T.
\end{align}
Thus, there exists a constant $c_f>0$ such that
\begin{align}
\mathbf{E}[\|\bsx_k\|^2_{\bsK}+W_{4,k}]\le nc_f,~\forall k\le T.\label{sguT:thm-sg-sm-bounded}
\end{align}

From \eqref{sguT:vkLya3.2} and \eqref{sguT:vkLya3.1}, we have
\begin{align}\label{sgu:vkLya3.2-bound}
0\le2\kappa_6(W_{1,k}+W_{2,k})\le\breve{W}_k\le2\kappa_5(W_{1,k}+W_{2,k}).
\end{align}

From \eqref{nonconvex:lemma:lipschitz2}, we have
\begin{align}
&\|\bar{\bsg}^0_{k}\|^2
=n\|\nabla f(\bar{x}_k)\|^2
\le2L_fn(f(\bar{x}_k)-f^*)=2L_fW_{4,k}.\label{sgu:vkLya-2gbar0k}
\end{align}

Denote $\breve{z}_k=\mathbf{E}[\breve{W}_k]$. From \eqref{sguT:vkLya2-1-3-speed} and  \eqref{sguT:thm-sg-sm-bounded}--\eqref{sgu:vkLya-2gbar0k}, we have
\begin{align}\label{sguT:vkLya4-bound}
\breve{z}_{k+1}\le(1-a_1)\breve{z}_k+a_2\eta^2,~\forall k\le T,
\end{align}
where $a_1=\min\{\varepsilon_{3},~\varepsilon_{4}\}/\kappa_5$ and $
a_2=n(4\varepsilon_{5}L_fc_f+2\varepsilon_{5}L_f^2c_f+(\varepsilon_{5}+3)\sigma^2)$.

From \eqref{sguT:vkLya2-pl-r1.1T}, we have
\begin{align}\label{sguT:vkLya2-a1-bounded}
a_1\le\frac{\varepsilon_4}{\kappa_5}\le\frac{1}{40}.
\end{align}

From \eqref{sguT:vkLya4-bound} and \eqref{sguT:vkLya2-a1-bounded}, we have
\begin{align*}
\breve{z}_{k+1}\le(1-a_1)^{k+1}\breve{z}_0+\frac{a_2\eta^2}{a_1},~\forall k\le T,
\end{align*}
which yields \eqref{sguT:thm-sg-equ2.1bounded}.

From \eqref{sguT:v4k-speed} and \eqref{nonconvex:gg3}, we have
\begin{align}\label{sguT:thm-sg-equ2bounded-proof}
\mathbf{E}_{\mathfrak{F}_k}[W_{4,k+1}]
&\le \Big(1-\frac{\nu\eta}{2}\Big)W_{4,k}+\frac{1}{2}\eta L_f^2\|\bsx_k\|^2_{\bsK}+L_f\sigma^2\eta^2\nonumber\\
&\le \Big(1-\frac{\nu\eta}{2}\Big)^{k+1}W_{4,0}+\frac{1}{\nu}(L_f^2\|\bsx_k\|^2_{ \bsK}+2L_f\sigma^2\eta).
\end{align}

Noting $\eta=1/(T+1)^\theta$, from \eqref{sguT:thm-sg-equ2bounded-proof}, \eqref{sgu:lemma:exponential-equ}, and \eqref{sguT:thm-sg-equ2.1bounded}, we have \eqref{sguT:thm-sg-equ2bounded}.

\subsection{Proof of Theorem~\ref{sgu:thm-sg-diminishingt}}\label{sgu:proof-thm-sg-diminishingt}
In addition to the notations defined in Appendix~\ref{sguT:proof-thm-sg-smT},
we also denote the following notations.
\begin{align*}
\tilde{c}_0(\kappa_1,\kappa_2)&=\max\Big\{4\varepsilon_{11},~\varepsilon_{6},
~\frac{\varepsilon_{10}}{\varepsilon_4}\Big\},\\
\hat{c}_2(\kappa_1)&=\min\Big\{\frac{\varepsilon_1}{\varepsilon_2},
~\frac{\varepsilon_8}{\varepsilon_9},~\frac{1}{5}\Big\},\\
\hat{c}_3(\kappa_0,\kappa_1,\kappa_2)&=\max\Big\{\frac{\tilde{c}_0(\kappa_1,\kappa_2)}{\kappa_0},
~\frac{8L_f\kappa_3}{\nu\kappa_2},~\frac{16L_f(\kappa_3-1)}{\nu\kappa_0\kappa_2}\Big\},\\
\tilde{\sigma}^2&=2L_ff^*-2L_f\frac{1}{n}\sum_{i=1}^{n}f_i^*,\\
\varepsilon_8&=\kappa_1\rho_2(L)-1,\\
\varepsilon_9&=\frac{1}{2}(3\kappa_1+2)\kappa_1\rho(L^2)+\rho(L)+1,\\
\varepsilon_{10}&=\kappa_2(\kappa_3-1)+\kappa_1\kappa_2+\kappa_3-1+3\kappa_2^2,\\
\varepsilon_{11}&=\kappa_2L_f+(2\kappa_3-1+\kappa_2(10\kappa_3-4))L_f^2,\\
\varepsilon_{12}&=3+L_f+\frac{\kappa_3L_f^2}{\kappa_0\kappa_2t_1}
+\frac{2\kappa_4L_f^2}{\kappa_0^2t_1^{2}}+\frac{2+2\kappa_3L_f^2}{\kappa_0t_1^{2}}\\
&\quad+\frac{(\kappa_3-1)L_f^2}{\kappa_0^2\kappa_2t_1^{3}}
+\frac{(\kappa_3-1)L_f^2}{\kappa_0^2t_1^{4}}\Big(\frac{2}{\kappa_0}+2\Big),\\
\varepsilon_{13}&=\frac{\kappa_0\kappa_3}{\kappa_2^2}
+\frac{\kappa_3-1}{\kappa_2^2t_1^{2}},\\
\varepsilon_{14}&=\varepsilon_{12}\sigma^2+\varepsilon_{13}\tilde{\sigma}^2,\\
\varepsilon_{15}&=\frac{1}{\kappa_5}\min\Big\{\frac{\varepsilon_3\kappa_0t_1}{\kappa_2},
~\frac{\varepsilon_4\kappa_0t_1}{2\kappa_2},~\frac{\nu}{8}\Big\},\\
\varepsilon_{16}&=\frac{4L_f\sigma^2\kappa_2^2}{\kappa_0^2(\frac{\nu\kappa_2}{2\kappa_0}-1)}.
\end{align*}

To prove Theorem~\ref{sgu:thm-sg-diminishingt}, we  need the following lemma.
\begin{lemma}\label{sgu:lemma:sg2}
Suppose Assumptions~\ref{sgd:assgraph}--\ref{sgd:ass:stochastic-grad:variance} hold. Suppose $\alpha_k=\kappa_1\beta_k$, $\beta_k=\kappa_0(k+t_1)$, and $\eta_k=\kappa_2/\beta_k$, where $\kappa_0\ge \tilde{c}_0(\kappa_1,\kappa_2)/t_1$, $\kappa_1>c_1$, $\kappa_2\in(0,\hat{c}_2(\kappa_1))$, and $t_1\ge 1$. Then, for any $ k\in\mathbb{N}_0$ the following holds for Algorithm~\ref{sguT:algorithm-sg}
\begin{subequations}
\begin{align}
\mathbf{E}_{\mathfrak{F}_k}[W_{k+1}]
&\le W_{k}-\varepsilon_3\|\bsx_k\|^2_{\bsK}
-\frac{1}{2}\varepsilon_4\Big\|\bm{v}_k+\frac{1}{\beta_k}\bsg_{k}^0\Big\|^2_{\bsK}
-\frac{1}{4}\eta_k\|\bar{\bsg}^0_{k}\|^2\nonumber\\
&\quad
+2L_fb_{8,k}\eta_k^2W_{4,k}+n\varepsilon_{14}\eta_k^2,\label{sgu:vkLya2}\\
\mathbf{E}_{\mathfrak{F}_k}[\breve{W}_{k+1}]
&\le \breve{W}_{k}-\varepsilon_3\|\bsx_k\|^2_{\bsK}
-\frac{1}{2}\varepsilon_4\Big\|\bm{v}_k+\frac{1}{\beta_k}\bsg_{k}^0\Big\|^2_{\bsK}
+n\varepsilon_{14}\eta_k^2\nonumber\\
&\quad+2\varepsilon_{12}L_f^2\eta_k^2\|\bsx_k\|^2_{\bsK}
+2(2\varepsilon_{12}+\varepsilon_{13})L_f\eta_k^2W_{4,k},\label{sgu:vkLya2-bounded}\\
\mathbf{E}_{\mathfrak{F}_k}[W_{4,k+1}]
&\le W_{4,k}-\frac{\eta_k}{4}\|\bar{\bsg}_{k}^0\|^2
+\|\bsx_k\|^2_{\frac{1}{2}L_f^2\eta_k\bsK}
+\eta^2_kL_f\sigma^2,\label{sgu:v4kspeed-diminishing}
\end{align}
\end{subequations}
where $b_{8,k}=\kappa_3\frac{\omega_k}{\eta_k^2}+(\kappa_3-1)\frac{\omega_k^2}{\eta_k^2}$.
\end{lemma}
\begin{proof}
(i) We have
\begin{align}
&\|\bsg^0_{k}\|^2
=\sum_{i=1}^{n}\|\nabla f_i(\bar{x}_k)\|^2
\le\sum_{i=1}^{n}2L_f(f_i(\bar{x}_k)-f_i^*)
=2L_fn(f(\bar{x}_k)-f^*)+n\tilde{\sigma}^2,\label{sgu:vkLya-2g0k}
\end{align}
where the inequality holds due to \eqref{nonconvex:lemma:lipschitz2}.

From  the Cauchy--Schwarz inequality, \eqref{sguT:gg}, and \eqref{sgu:vkLya-2g0k}, we have
\begin{align}
\|\bsg^0_{k+1}\|^2
&=\|\bsg^0_{k+1}-\bsg^0_{k}+\bsg^0_{k}\|^2
\le2(\|\bsg^0_{k+1}-\bsg^0_{k}\|^2+\|\bsg^0_{k}\|^2)\nonumber\\
&\le 2(\eta^2_kL_f^2\|\bar{\bsg}^u_k\|^2+2L_fW_{4,k}+n\tilde{\sigma}^2).\label{sgu:vkLya-2}
\end{align}

From \eqref{sguT:v1k}, \eqref{sguT:v2k}, \eqref{sguT:v3k}, \eqref{sguT:v4k}, \eqref{sguT:vkLya-3}, \eqref{sgu:vkLya-2}, $\alpha_k=\kappa_1\beta_k$,  and $\eta_k=\kappa_2/\beta_k$, we have
\begin{align}\label{sgu:vkLya-4}
\mathbf{E}_{\mathfrak{F}_k}[W_{k+1}]
&\le W_{k}-\|\bsx_k\|^2_{\eta_k\bsM_{3,k}-\eta_k^2\bsM_{4,k}-\frac{1}{2}\kappa_1\kappa_2\omega_k
+\eta_k\omega_k\bsM_{5,k}
-\eta_k^2\omega_k\bsM_{6,k}}\nonumber\\
&\quad-\Big\|\bm{v}_k+\frac{1}{\beta_k}\bsg_{k}^0\Big\|^2_{b^0_{4,k}\bsK}
-\eta_kb_{5,k}\|\bar{\bsg}_{k}\|^2-\frac{1}{4}\eta_k\|\bar{\bsg}^0_{k}\|^2\nonumber\\
&\quad+\eta_k^2(b_{6,k}+b_{7,k}n)\sigma^2
+\eta_k^2b_{8,k}(2L_fW_{4,k}+n\tilde{\sigma}^2),
\end{align}
where
\begin{align*}
\bsM_{3,k}&=(\alpha_k-\beta_k)\bsL-\frac{1}{2}(2+3L_f^2)\bsK,\\
\bsM_{4,k}&=\beta^2_k\bsL+(2\alpha^2_k+\beta^2_k)\bsL^2+4L_f^2\bsK,\\
\bsM_{5,k}&=\alpha_k\bsL-\frac{1}{2}(1+L_f^2)\bsK,\\
\bsM_{6,k}&=\frac{3}{2}\alpha^2_k\bsL^2+\beta^2_k(\bsL+\kappa_1\bsL^2)+\frac{5}{2}L_f^2\bsK,\\
b^0_{4,k}&=\frac{1}{2}\eta_k(2\beta_k-\kappa_3)
-\frac{5}{2}\eta^2_k\beta^2_k-\frac{1}{2}\omega_k\eta_k(\kappa_3-1)
-\frac{1}{2}\omega_k(\eta_k\alpha_k+\kappa_3-1+3\eta_k^2\beta_k^2),\\
b_{5,k}&=\frac{1}{4}-\eta_kb_{6,k},\\
b_{6,k}&=L_f+\frac{1}{\beta^2_k\eta_k}\kappa_3L_f^2+\frac{2}{\beta^2_k}\kappa_4L_f^2
+2\kappa_3L_f^2\omega_k
+\omega_k\Big(\frac{1}{\beta^2_k\eta_k}+\frac{2}{\beta^2_k}+2\omega_k\Big)(\kappa_3-1)L_f^2,\\
b_{7,k}&=3+2\omega_k.
\end{align*}

From \eqref{nonconvex:KL-L-eq2}, $\alpha_k=\kappa_1\beta_k$, $\kappa_1>1$, $\beta_k\ge\kappa_0t_1\ge\tilde{c}_0(\kappa_1,\kappa_2)\ge\varepsilon_6\ge(2+3L_f^2)/2$, and $\eta_k=\kappa_2/\beta_k$, we have
\begin{align}\label{sgu:vkLya-M1}
\eta_k\bsM_{3,k}\ge\varepsilon_1\kappa_2\bsK.
\end{align}

From \eqref{nonconvex:KL-L-eq2}, $\alpha_k=\kappa_1\beta_k$, $\beta_k\ge(2+3L_f^2)/2>2L_f$, and $\eta_k=\kappa_2/\beta_k$, we have
\begin{align}\label{sgu:vkLya-M2}
\eta_k^2\bsM_{4,k}\le\varepsilon_2\kappa_2^2\bsK.
\end{align}

From \eqref{nonconvex:KL-L-eq2}, $\alpha_k=\kappa_1\beta_k$,  $\beta_k\ge(2+3L_f^2)/2>(1+L_f^2)/2$, and $\eta_k=\kappa_2/\beta_k$, we have
\begin{align}\label{sgu:vkLya-M4}
\eta_k\bsM_{5,k}\ge\varepsilon_8\kappa_2\bsK.
\end{align}

From \eqref{nonconvex:KL-L-eq2}, $\alpha_k=\kappa_1\beta_k$, $\beta_k>2L_f>\sqrt{10}L_f/2$, and $\eta_k=\kappa_2/\beta_k$, we have
\begin{align}\label{sgu:vkLya-M5}
\eta_k^2\bsM_{6,k}\le\varepsilon_9\kappa_2^2\bsK.
\end{align}

From $\alpha_k=\kappa_1\beta_k$, $\beta_k\ge\kappa_3$, and $\eta_k=\kappa_2/\beta_k$, we have
\begin{align}\label{sgu:vkLya-b1}
b^0_{4,k}\ge&b_{4,k},
\end{align}
where $b_{4,k}=\varepsilon_4-\frac{1}{2}\omega_k\eta_k(\kappa_3-1)
-\frac{1}{2}\omega_k(\kappa_1\kappa_2+\kappa_3-1+3\kappa_2^2)$.

From $\kappa_1>c_1=1/\rho_2(L)+1$, we have
\begin{align}
\varepsilon_1>0~\text{and}~
\varepsilon_8>0.\label{sgu:varepsilon1and8}
\end{align}

From \eqref{sgu:varepsilon1and8} and $\kappa_2\in(0,\min\{\frac{\varepsilon_1}{\varepsilon_2},~\frac{\varepsilon_8}{\varepsilon_9},~\frac{1}{5}\})$, we have
\begin{subequations}
\begin{align}
&\varepsilon_3>0,\label{sgu:kappa2-1}\\
&\varepsilon_8\kappa_2-\varepsilon_9\kappa_2^2>0,\label{sgu:kappa2-2}\\
&\varepsilon_4>0.\label{sgu:kappa2-3}
\end{align}
\end{subequations}

From $\beta_k=\kappa_0(k+t_1)$, we have
\begin{align}\label{sgu:omegak}
\omega_k=\frac{1}{\beta_{k}}-\frac{1}{\beta_{k+1}}
=\frac{1}{\kappa_0}\Big(\frac{1}{k+t_1}-\frac{1}{k+t_1+1}\Big)
=\frac{1}{\kappa_0(k+t_1)(k+t_1+1)}
\le\frac{\kappa_0}{\beta_k^2}.
\end{align}

From \eqref{sgu:kappa2-1}--\eqref{sgu:omegak}, and $\kappa_0\ge
\max\{\frac{4\varepsilon_{11}}{t_1},~\frac{\varepsilon_{10}}{\varepsilon_4t_1}\}$, we have
\begin{subequations}
\begin{align}
&b_{4,k}\ge\varepsilon_4-\frac{\varepsilon_{10}}{2\kappa_0t_1^{2}}
\ge\varepsilon_4-\frac{\varepsilon_{10}}{2\kappa_0t_1}\ge\frac{1}{2}\varepsilon_4>0,
\label{sgu:varepsilon4}\\
&b_{5,k}\ge\frac{1}{4}-\frac{\varepsilon_{11}}{\kappa_0t_1}\ge0.
\label{sgu:varepsilon11}
\end{align}
\end{subequations}

From \eqref{sgu:omegak} and $\beta_k=\kappa_0(k+t_1)\ge\kappa_0t_1$, we have
\begin{subequations}
\begin{align}
&b_{6,k}+b_{7,k}\le\varepsilon_{12},\label{sgu:varepsilon12}\\
&b_{8,k}\le\varepsilon_{13}.\label{sgu:varepsilon13}
\end{align}
\end{subequations}

From \eqref{sgu:vkLya-4}--\eqref{sgu:vkLya-b1}, \eqref{sgu:kappa2-1}--\eqref{sgu:kappa2-3}, and \eqref{sgu:varepsilon4}--\eqref{sgu:varepsilon13}, we have \eqref{sgu:vkLya2}.

(ii) Similarly to the way to get \eqref{sgu:vkLya2}, we have
\begin{align}
\mathbf{E}_{\mathfrak{F}_k}[\breve{W}_{k+1}]
&\le \breve{W}_{k}-\varepsilon_3\|\bsx_k\|^2_{\bsK}
-\frac{1}{2}\varepsilon_4\Big\|\bm{v}_k+\frac{1}{\beta_k}\bsg_{k}^0\Big\|^2_{\bsK}
+\varepsilon_{12}\eta_k^2\|\bar{\bsg}_{k}\|^2\nonumber\\
&\quad+2L_f\varepsilon_{13}\eta_k^2W_{4,k}+n\varepsilon_{14}\eta_k^2,~\forall k\in\mathbb{N}_0,\label{sgu:vkLya2-bounded-1}
\end{align}

From \eqref{sgu:vkLya2-bounded-1}, \eqref{sgu:vkLya-2g0kbar}, and \eqref{sgu:vkLya-2gbar0k}, we have \eqref{sgu:vkLya2-bounded}.

(iii) From \eqref{sguT:v4k} and \eqref{sguT:vkLya-3}, we have
\begin{align}\label{sgu:v4kspeed-diminishing-1}
\mathbf{E}_{\mathfrak{F}_k}[W_{4,k+1}]
&\le  W_{4,k}-\frac{\eta_k}{4}\|\bar{\bsg}_{k}\|^2
+\|\bsx_k\|^2_{\frac{1}{2}L_f^2\eta_k\bsK}
-\frac{\eta_k}{4}\|\bar{\bsg}_{k}^0\|^2
+\eta^2_kL_f(\sigma^2+\|\bar{\bsg}_k\|^2).
\end{align}

From $0<\eta_k\le\kappa_2/(\kappa_0t_1)$ and $\kappa_0t_1\ge\tilde{c}_0(\kappa_1,\kappa_2)\ge4\varepsilon_{11}>4\kappa_2L_f$, we have
\begin{align}
\eta_k^2L_f<\frac{1}{4}\eta_k.\label{sgu:v4kspeed-diminishing-1.3}
\end{align}
From \eqref{sgu:v4kspeed-diminishing-1} and \eqref{sgu:v4kspeed-diminishing-1.3}, we have
\eqref{sgu:v4kspeed-diminishing}.
\end{proof}

Now we are ready to prove Theorem~\ref{sgu:thm-sg-diminishingt}.

From $t_1>\hat{c}_3(\kappa_0,\kappa_1,\kappa_2)\ge \tilde{c}_0(\kappa_1,\kappa_2)/\kappa_0$, we have
\begin{align*}
\kappa_0>\frac{\tilde{c}_0(\kappa_1,\kappa_2)}{t_1}.
\end{align*}
Thus, all conditions needed in Lemma~\ref{sgu:lemma:sg2} are satisfied, so \eqref{sgu:vkLya2}--\eqref{sgu:v4kspeed-diminishing} hold.

From \eqref{sgu:vkLya2} and \eqref{nonconvex:gg3}, we have
\begin{align}\label{sgu:vkLya2-pl-0}
\mathbf{E}_{\mathfrak{F}_k}[W_{k+1}]
&\le W_{k}-\varepsilon_3\|\bsx_k\|^2_{\bsK}
-\frac{1}{2}\varepsilon_4\Big\|\bm{v}_k+\frac{1}{\beta_k}\bsg_{k}^0\Big\|^2_{\bsK}
-\frac{\eta_k\nu }{2}W_{4,k}
+2L_fb_{8,k}\eta_k^2W_{4,k}
+n\varepsilon_{14}\eta_k^2\nonumber\\
&= W_{k}-\varepsilon_3\|\bsx_k\|^2_{\bsK}
-\frac{1}{2}\varepsilon_4\Big\|\bm{v}_k+\frac{1}{\beta_k}\bsg_{k}^0\Big\|^2_{\bsK}
+n\varepsilon_{14}\eta_k^2\nonumber\\
&~~~-2\Big(\frac{1 }{4}-\frac{1}{\nu}L_fb_{8,k}\eta_k\Big)\nu\eta_kW_{4,k}
,~\forall k\in\mathbb{N}_0.
\end{align}

From $t_1>\hat{c}_3(\kappa_0,\kappa_1,\kappa_2)\ge 8L_f\kappa_3/(\nu\kappa_2)$, we have
\begin{align}\label{sgu:b6k-1}
\frac{1}{4}-\frac{L_f\kappa_3}{\nu\kappa_2t_1}\ge\frac{1}{8}.
\end{align}

From \eqref{sgu:omegak}, \eqref{sgu:b6k-1}, and $\kappa_0>\tilde{c}_0(\kappa_1,\kappa_2)/t_1\ge16L_f(\kappa_3-1)/(\nu\kappa_2t_1)$, we have
\begin{align}
\frac{1}{4}-\frac{1}{\nu}L_fb_{8,k}\eta_k
&\ge\frac{1}{4}-\frac{L_f\kappa_0\kappa_3}{\nu\kappa_2\beta_k}
-\frac{L_f\kappa_0^2(\kappa_3-1)}{\nu\kappa_2\beta_k^3}\nonumber\\
&\ge\frac{1}{4}-\frac{L_f\kappa_3}{\nu\kappa_2t_1}
-\frac{L_f(\kappa_3-1)}{\nu\kappa_2\kappa_0t_1^{3}}
\ge\frac{1}{8}-\frac{L_f(\kappa_3-1)}{\nu\kappa_2\kappa_0t_1}\ge\frac{1}{16}.
\label{sgu:b6k}
\end{align}

From \eqref{sgu:vkLya2-pl-0}, \eqref{sguT:vkLya3}, and \eqref{sguT:vkLya3.1}, we have
\begin{align}\label{sgu:vkLya2-pl}
\mathbf{E}_{\mathfrak{F}_k}[W_{k+1}]
&\le W_{k}-\frac{\eta_k}{\kappa_5}\min\Big\{\frac{\varepsilon_3}{\eta_k},
~\frac{\varepsilon_4}{2\eta_k},~\frac{\nu}{8}\Big\}W_{k}+n\varepsilon_{14}\eta_k^2\nonumber\\
&\le W_{k}-\varepsilon_{15}\eta_kW_{k}+n\varepsilon_{14}\eta_k^2,~\forall k\in\mathbb{N}_0.
\end{align}

Denote $z_k=\mathbf{E}[W_k]$, $r_{1,k}=\varepsilon_{15}\eta_k$, and $r_{2,k}=n\varepsilon_{14}\eta_k^2$, then from \eqref{sgu:vkLya2-pl}, we have
\begin{align}
z_{k+1}
\le (1-r_{1,k})z_k+r_{2,k},~\forall k\in\mathbb{N}_0.
\label{sgu:vkLya2-pl-z}
\end{align}

From \eqref{sgu:step:eta1t1}, we have
\begin{subequations}
\begin{align}
r_{1,k}&=\eta_k\varepsilon_{15}
=\frac{a_3}{k+t_1},\label{sgu:vkLya2-pl-r1}\\
r_{2,k}&=\eta_k^2\varepsilon_{14}n\sigma^2
=\frac{a_4}{(k+t_1)^{2}},\label{sgu:vkLya2-pl-r2}
\end{align}
\end{subequations}
where $a_3=\kappa_2\varepsilon_{15}/\kappa_0$ and $
a_4=n\kappa_2^2\varepsilon_{14}/\kappa^2_0$.

From \eqref{sguT:vkLya2-pl-r1.1T}, we have
\begin{align}\label{sgu:vkLya2-pl-r1.1}
r_{1,k}\le\frac{\varepsilon_4}{2\kappa_5}\le\frac{1}{80}.
\end{align}

Then, from \eqref{sgu:vkLya2-pl-z}--\eqref{sgu:vkLya2-pl-r1.1} and \eqref{zerosg:serise:lemma:sequence-equ5}, we have
\begin{align}\label{sgu:vkLya2-pl-theta1}
z_{k}\le\phi_1(k,t_1,a_3,a_4,z_0),~\forall k\in\mathbb{N}_+,
\end{align}
where the function $\phi_1$ is defined in \eqref{zerosg:serise:lemma:sequence-equ5-phi3}.

From $\kappa_0\ge\hat{c}_0\nu\kappa_2/4$, we have
\begin{align}\label{sgu:vkLya2-pl-phi2}
\phi_1(k,t_1,a_3,a_4,2,z_0)=
\begin{cases}
  \mathcal{O}(\frac{n}{k}), & \mbox{if } a_3>1, \\
  \mathcal{O}(\frac{n\ln(k-1)}{k}), & \mbox{if } a_3=1, \\
  \mathcal{O}(\frac{n}{k^{a_3}}), & \mbox{if } a_3<1,
\end{cases}
\end{align}
From \eqref{sgu:vkLya2-pl-theta1}, \eqref{sgu:vkLya2-pl-phi2}, and \eqref{sguT:vkLya5}, we know that there exists a constant $c_f>0$ such that
\begin{align}
\mathbf{E}[\|\bsx_k\|^2_{\bsK}+W_{4,k}]\le nc_f.\label{sgu:thm-sg-sm-bounded}
\end{align}

From \eqref{sgu:vkLya2-bounded}, \eqref{sgu:thm-sg-sm-bounded}, \eqref{sgu:vkLya3.2-bound}, and \eqref{sgu:step:eta1t1}, we have
\begin{align}\label{sgu:vkLya4-bound}
\breve{z}_{k+1}\le(1-a_5)\breve{z}_k+\frac{a_6}{(t+t_1)^{2}},
\end{align}
where $a_5=\min\{\varepsilon_{3},~\varepsilon_{4}/2\}/\kappa_5$ and $
a_6=n(2\varepsilon_{12}L_f^2c_f+2(2\varepsilon_{12}+\varepsilon_{13})L_fc_f+\varepsilon_{14})\kappa_2^2/\kappa_0^2$.

From \eqref{sguT:vkLya2-pl-r1.1T}, we have
\begin{align}\label{sgu:vkLya2-a1-bounded}
a_5\le\frac{\varepsilon_4}{2\kappa_5}\le\frac{1}{80}.
\end{align}

From \eqref{sgu:kappa2-1} and \eqref{sgu:kappa2-3}, we know that
\begin{align}\label{sgu:vkLya2-pl-a1a2-bounded}
a_5>0~\text{and}~a_6>0.
\end{align}

From \eqref{sgu:vkLya4-bound}--\eqref{sgu:vkLya2-pl-a1a2-bounded} and \eqref{zerosg:serise:lemma:sequence-equ6}, we have
\begin{align}\label{sgu:lemma:sequence-equ6-bounded}
\breve{z}_{k}\le\phi_2(k,t_1,a_5,a_6,\breve{z}_{0})=\mathcal{O}(\frac{n}{k^{2}}),~\forall k\in\mathbb{N}_+,
\end{align}
where the function $\phi_2$ is defined in \eqref{zerosg:serise:lemma:sequence-equ6-phi4}.

From \eqref{zerosg:serise:lemma:sequence-equ6-phi4}, \eqref{sgu:vkLya3.2-bound}, and \eqref{sgu:lemma:sequence-equ6-bounded}, we have
\begin{align}\label{sgu:vkLya4-bound-brevez}
&\mathbf{E}[\|\bsx_k\|^2_{\bsK}]\le\frac{1}{\kappa_6}\breve{z}_{k}
\le\frac{1}{\kappa_6}\phi_2(k,t_1,a_5,a_6,\breve{z}_{0})=\mathcal{O}(\frac{n}{k^{2}}).
\end{align}

From \eqref{sgu:vkLya4-bound-brevez}, we have \eqref{sgu:thm-sg-diminishing-equ2.1bounded}.

From \eqref{sgu:v4kspeed-diminishing} and \eqref{nonconvex:gg3}, we have
\begin{align}
&\mathbf{E}[W_{4,k+1}]\le (1-\frac{\nu}{2}\eta_k)\mathbf{E}[W_{4,k}]
+\|\bsx_k\|^2_{\frac{1}{2}L_f^2\eta_k\bsK}+L_f\sigma^2\eta^2_k.
\label{sgu:v4kspeed-diminishing-3}
\end{align}

From $\kappa_0<\nu\kappa_2/4$, we have
\begin{align}\label{sgu:vkLya2-pl-a1-3-bounded}
\frac{\nu\kappa_2}{2\kappa_0}>2.
\end{align}

Similar to the way to prove \eqref{zerosg:serise:lemma:sequence-equ5}, from \eqref{sgu:vkLya4-bound-brevez}--\eqref{sgu:vkLya2-pl-a1-3-bounded}, we have
\begin{align}\label{sgu:thm-sg-diminishing-equ2bounded-proof}
\mathbf{E}[f(\bar{x}_{T})-f^*]
\le\frac{\varepsilon_{16}}{n(T+t_1)}+\mathcal{O}(\frac{1}{(T+t_1)^{2}}),
\end{align}
where $\varepsilon_{16}$ is determined by the last terms in \eqref{zerosg:serise:lemma:sequence-equ5-phi3} and \eqref{sgu:v4kspeed-diminishing-3}.

From $\kappa_0\ge\hat{c}_0\nu\kappa_2/4$, we have
\begin{align}\label{sgu:thm-sg-diminishing-equ2bounded-proof2}
\varepsilon_{16}&=\frac{4L_f\sigma^2\kappa_2^2}{\kappa_0^2(\frac{\nu\kappa_2}{2\kappa_0}-1)}
\le\frac{4L_f\sigma^2\kappa_2^2}{\kappa_0^2(\frac{\nu\kappa_2}{2\kappa_0}
-\frac{\nu\kappa_2}{4\kappa_0})}=\frac{16L_f\sigma^2\kappa_2}{\nu\kappa_0}\le\frac{64L_f\sigma^2}{\hat{c}_0\nu^2}.
\end{align}

From \eqref{sgu:thm-sg-diminishing-equ2bounded-proof} and \eqref{sgu:thm-sg-diminishing-equ2bounded-proof2}, we have \eqref{sgu:thm-sg-diminishing-equ2bounded}.

\subsection{Proof of Theorem~\ref{sg:thm-sg-fixed-a5}}\label{sg:proof-thm-sg-fixed-a5}
In addition to the notations defined in Appendices~\ref{sguT:proof-thm-sg-smT} and \ref{sguT:proof-thm-sg-diminishingT},
we also denote the following notations.
\begin{align*}
\varepsilon&=\frac{1}{\kappa_5}\min\Big\{\frac{\varepsilon_3}{\eta},~\frac{\varepsilon_4}{\eta},
~\frac{\nu}{2}\Big\},\\
c_4&=\frac{W_0}{n\kappa_6},\\
c_5&=\frac{\varepsilon_5+3n}{n\varepsilon\kappa_6}.
\end{align*}

From the conditions in Theorem~\ref{sg:thm-sg-fixed-a5}, we know that \eqref{sguT:vkLya2-plT-constant} holds. Thus,
\begin{align}\label{sg:sgproof-vkLya4-fixed-sg-ft-a5}
\mathbf{E}_{\mathfrak{F}_k}[W_{k+1}]
\le W_{k}-\eta\varepsilon W_{k}
+(\varepsilon_5+3n)\sigma^2\eta^2.
\end{align}

Similar to the way to get \eqref{sguT:vkLya2-pl-r1.1T}, we have
\begin{align}\label{sg:sgproof-epsilon-a5}
0<\eta\varepsilon<1.
\end{align}

From \eqref{sg:sgproof-vkLya4-fixed-sg-ft-a5} and \eqref{sg:sgproof-epsilon-a5}, we have
\begin{align}
\mathbf{E}[W_{k+1}]
&\le (1-\eta\varepsilon)\mathbf{E}[W_k]+(\varepsilon_5+3n)\sigma^2\eta^2\nonumber\\
&\le(1-\eta\varepsilon)^{k+1}W_0
+(\varepsilon_5+3n)\sigma^2\eta^2\sum_{\tau=0}^{k}(1-\eta\varepsilon)^\tau\nonumber\\
&\le(1-\eta\varepsilon)^{k+1}W_0
+\frac{\eta(\varepsilon_5+3n)\sigma^2}{\varepsilon}.\label{sg:sgproof-vkLya4-a5}
\end{align}

Hence, \eqref{sg:sgproof-vkLya4-a5} and \eqref{sguT:vkLya5} give \eqref{sg:thm-sg-fixed-equ1-a5}.

\subsection{Proof of Theorem~\ref{sg:thm-sg-fixed}}\label{sg:proof-thm-sg-fixed}
In addition to the notations defined in Appendices~\ref{sguT:proof-thm-sg-diminishingT}, \ref{sguT:proof-thm-sg-smT}, and \eqref{sg:proof-thm-sg-fixed-a5},
we also denote the following notations.
\begin{align*}
\breve{c}_0(\kappa_1,\kappa_2)&=\max\{4\kappa_2\varepsilon_5,~\breve{\varepsilon}_{6}\},\\
\breve{\varepsilon}_6&=\max\{1+3L_f^2,~\kappa_3\},\\
\breve{c}_5&=\frac{3+5\eta}{\varepsilon\kappa_6}
\end{align*}

Without unbiased assumption, we know that \eqref{sguT:v2k2} still holds.
Similar to the way to get \eqref{sguT:v1k}, \eqref{sguT:v3k2}, and \eqref{sguT:v4k},  we have
\begin{subequations}
\begin{align}
\mathbf{E}_{\mathfrak{F}_k}[W_{1,k+1}]
&\le W_{1,k}-\|\bsx_k\|^2_{\eta\alpha\bsL-\frac{\eta}{2}\bsK
-\frac{3\eta^2\alpha^2}{2}\bsL^2-\eta(1+3\eta)L_f^2\bsK}\nonumber\\
&\quad-\eta\beta\bsx^\top_k\bsK\Big(\bm{v}_k+\frac{1}{\beta}\bsg_k^0\Big)
+\Big\|\bm{v}_k+\frac{1}{\beta}\bsg_k^0\Big\|^2_{\frac{3\eta^2\beta^2}{2}\bsK}
+\eta(1+3\eta)n\sigma^2,\label{sg:v1k}\\
\mathbf{E}_{\mathfrak{F}_k}[W_{3,k+1}]
&\le W_{3,k}-\eta\alpha\bm{x}_k^\top\bsL\Big(\bm{v}_k+\frac{1}{\beta}\bsg_{k}^0\Big)
+\|\bm{x}_k\|^2_{\eta(\beta\bsL+\frac{1}{2}\bsK)
+\eta^2(\frac{\alpha^2}{2}-\alpha\beta+\beta^2)\bsL^2+\eta(1+2\eta)L_f^2\bsK}\nonumber\\
&\quad+\frac{\eta}{2\beta^2}(1+3\eta)L_f^2
\mathbf{E}_{\mathfrak{F}_k}[\|\bar{\bsg}^u_{k}\|^2]+\eta(1+2\eta)n\sigma^2
-\Big\|\bm{v}_k+\frac{1}{\beta}\bsg_{k}^0\Big\|^2_{\eta(\beta-\frac{1}{2}
-\eta\beta^2)\bsK},\label{sg:v3k}\\
\mathbf{E}_{\mathfrak{F}_k}[W_{4,k+1}]
&\le W_{4,k}-\frac{\eta}{4}(1-2\eta L_f)\mathbf{E}_{\mathfrak{F}_k}[\|\bar{\bsg}^u_{k}\|^2]
+\|\bsx_k\|^2_{\eta L_f^2\bsK}-\frac{\eta}{4}\|\bar{\bsg}_{k}^0\|^2+n\sigma^2\eta.\label{sg:v4k}
\end{align}
\end{subequations}

Then, similar to the way to get \eqref{sguT:vkLya2}, from \eqref{sguT:v2k2} and \eqref{sg:v1k}--\eqref{sg:v4k}, we have
\begin{align}\label{sg:sgproof-vkLya2}
\mathbf{E}_{\mathfrak{F}_k}[W_{k+1}]
&\le   W_{k}-\|\bsx_k\|^2_{\varepsilon_3\bsK}
-\Big\|\bm{v}_k+\frac{1}{\beta}\bsg_{k}^0\Big\|^2_{\varepsilon_4\bsK}-\frac{1}{4}\eta\|\bar{\bsg}^0_{k}\|^2
+\eta(3+5\eta)n\sigma^2.
\end{align}

Then, similar to the way to get \eqref{sg:thm-sg-fixed-equ1-a5}, from \eqref{sg:sgproof-vkLya2} and \eqref{nonconvex:gg3}, we have \eqref{sg:thm-sg-fixed-equ1}.

\end{document}